\documentclass[11pt,a4paper, envcountsame,mathserif]{amsart}
\usepackage[usenames,dvipsnames]{color}

\usepackage[utf8]{inputenc}	
\usepackage{pdfpages}

 \usepackage[colorlinks,citecolor=blue,urlcolor=blue, linkcolor=blue, backref]{hyperref}

\usepackage[capitalise]{cleveref}		

 \usepackage{amsmath, amssymb, xspace}
 \usepackage{graphicx}

 \usepackage[all]{xy}

\usepackage{tikz}
\usetikzlibrary{arrows,snakes,positioning,backgrounds,shadows}
\usepackage{stmaryrd}

\newtheorem{theorem}{Theorem}[section]

\newtheorem{fact}[theorem]{Fact}

\newtheorem{proposition}[theorem]{Proposition}

\newtheorem{prop}[theorem]{Proposition}

\newtheorem{claim}[theorem]{Claim}

\newtheorem{lemma}[theorem]{Lemma}		

\newtheorem{corollary}[theorem]{Corollary}

\newtheorem{question}[theorem]{Question}

\theoremstyle{definition}
\newtheorem{definition}[theorem]{Definition}
\newtheorem{example}[theorem]{Example}
\newtheorem{axiom}[theorem]{Axioms}
\newtheorem{remark}[theorem]{Remark}

\newcommand{\NN}{{\mathbb{N}}}

\newcommand{\R}{{\mathbb{R}}}
\newcommand{\QQ}{{\mathbb{Q}}}
\newcommand{\ZZ}{{\mathbb{Z}}}
\newcommand{\sub}{\subseteq}
\newcommand{\sN}[1]{_{#1\in \NN}}
\newcommand{\uhr}[1]{\! \upharpoonright_{#1}}
\newcommand{\ML}{Martin-L{\"o}f}
\newcommand{\SI}[1]{\Sigma^0_{#1}}
\newcommand{\PI}[1]{\Pi^0_{#1}}
\newcommand{\PPI}{\PI{1}}

\newcommand{\bi}{\begin{itemize}}
\newcommand{\ei}{\end{itemize}}
\newcommand{\bc}{\begin{center}}
\newcommand{\ec}{\end{center}}

\newcommand{\ES}{\emptyset}

\newcommand{\tp}[1]{2^{#1}}
\newcommand{\ex}{\exists}
\newcommand{\fa}{\forall}

\newcommand{\la}{\langle}
\newcommand{\ra}{\rangle}
\newcommand{\Kuc}{Ku{\v c}era}
\newcommand{\seqcantor}{2^{ \NN}}

\newcommand{\cantor}{\seqcantor}

\newcommand{\Opcl}[1]{[#1]^\prec}
\newcommand{\leT}{\le_{\mathrm{T}}}

\newcommand{\MLR}{\mbox{\rm \textsf{MLR}}}

\newcommand{\n}{\noindent}

\newcommand{\vsps}{\vspace{3pt}}

\newcommand{\vsp}{\vspace{6pt}}
\newcommand{\leb}{\mathbf{\lambda}}

\newcommand{\sss}{\sigma}
\newcommand{\aaa}{\alpha}

\newcommand{\sq}{\sqsubseteq}
\renewcommand{\S}{S(\omega)}

\newcommand \seq[1]{{\left\langle{#1}\right\rangle}}

\newcommand\+[1]{\mathcal{#1}}

\newcommand{\wt}{\widetilde}
\newcommand{\ol}{\overline}

\newcommand{\ape}{\, \hat{\ } \, }

\newcommand{\lra}{\leftrightarrow}
\newcommand{\LR}{\Leftrightarrow}
\newcommand{\RA}{\Rightarrow}
\newcommand{\LA}{\Leftarrow}
\newcommand{\DA}{\downarrow}
\newcommand{\UA}{\uparrow}

\newcommand{\CCC}{\mathcal{C}}

\newcommand \DemBLR{\textup{Demuth}_{\textup{BLR}}}

\DeclareMathOperator{\Aut}{Aut}

\numberwithin{equation}{section}
\renewcommand{\hat}{\widehat}
\begin{document}

\title{Logic Blog 2020 (the 10th anniversary blog)}

 \author{Editor: Andr\'e Nies}

\maketitle


\setcounter{tocdepth}{1}
\tableofcontents

 {
The Logic   Blog is a shared platform for
\bi \item rapidly announcing  results and questions related to logic
\item putting up results and their proofs for further research
\item parking results for later use
\item getting feedback before submission to  a journal
\item foster collaboration.   \ei

Each year's   blog is    posted on arXiv a 2-3 months after the year has ended.
\vsp
\begin{tabbing}

 \href{http://arxiv.org/abs/2003.03361}{Logic Blog 2019} \ \ \ \   \= (Link: \texttt{http://arxiv.org/abs/2003.03361})  \\

 \href{http://arxiv.org/abs/1902.08725}{Logic Blog 2018} \ \ \ \   \= (Link: \texttt{http://arxiv.org/abs/1902.08725})  \\
 
 \href{http://arxiv.org/abs/1804.05331}{Logic Blog 2017} \ \ \ \   \= (Link: \texttt{http://arxiv.org/abs/1804.05331})  \\
 
 \href{http://arxiv.org/abs/1703.01573}{Logic Blog 2016} \ \ \ \   \= (Link: \texttt{http://arxiv.org/abs/1703.01573})  \\
 
  \href{http://arxiv.org/abs/1602.04432}{Logic Blog 2015} \ \ \ \   \= (Link: \texttt{http://arxiv.org/abs/1602.04432})  \\
  
  \href{http://arxiv.org/abs/1504.08163}{Logic Blog 2014} \ \ \ \   \= (Link: \texttt{http://arxiv.org/abs/1504.08163})  \\

   \href{http://arxiv.org/abs/1403.5719}{Logic Blog 2013} \ \ \ \   \= (Link: \texttt{http://arxiv.org/abs/1403.5719})  \\

    \href{http://arxiv.org/abs/1302.3686}{Logic Blog 2012}  \> (Link: \texttt{http://arxiv.org/abs/1302.3686})   \\

 \href{http://arxiv.org/abs/1403.5721}{Logic Blog 2011}   \> (Link: \texttt{http://arxiv.org/abs/1403.5721})   \\

 \href{http://dx.doi.org/2292/9821}{Logic Blog 2010}   \> (Link: \texttt{http://dx.doi.org/2292/9821})  
     \end{tabbing}

\vsp

\n {\bf How does the Logic Blog work?}

\vsp

\n {\bf Writing and editing.}  The source files are in a shared dropbox.
 Ask Andr\'e (\email{andre@cs.auckland.ac.nz})  in order    to gain access.

\vsp

\n {\bf Citing.}  Postings can be cited.  An example of a citation is:

\vsp

\n  H.\ Towsner, \emph{Computability of Ergodic Convergence}. In  Andr\'e Nies (editor),  Logic Blog, 2012, Part 1, Section 1, available at
\url{http://arxiv.org/abs/1302.3686}.}

%

%

 
The logic blog,  once it is on  arXiv,  produces citations e.g.\ on Google Scholar.
%

%

  \part{Group  theory and its connections to  logic}
   
%

  \section{Nies and Stephan:  non-word automatic free nil-2 groups} 
We follow up on a post on word automaticity of various nilpotent groups by the same guys last year \cite[Section~6]{LogicBlog:19}. To be word automatic means that  given an appropriate encoding of the elements by strings,  the domain and operations are recognizable by finite automata. Other terms in use  for this notion include FA-presentable, and even ``automatic".

 We supply a proof that examples of the last  of three types that were provided there are not word-automatic. The first two types were shown to be word automatic there. 
For an odd prime $p$, let $L_p$ be the free nilpotent-2 exponent $p$ group  of infinite rank. In the notation of \cite[Section~6]{LogicBlog:19} one  has $N\cong Q$;  let $y_{i,k}$ ($i< k$) be free generators of $N$; then  $L_p$ is given by the function $\phi (x_i, x_k) = y_{i,k}$ if $i<k$.  So we have   
  
  \bc $L_p = \la x_i, y_{i,k} \mid x_i^p, [x_i, x_k] y_{i,k}^{-1} , [y_{i,k}, x_r] (i< k) \ra$ \ec
thus the $y_{i,k}$ are actually redundant. This group in a sense generalises the Heisenberg group over the ring $\mathbb F_p$, which would be the case of two generators $x_0, x_1$.

\begin{theorem}  $L_p$ is not word automatic.\end{theorem}
\begin{proof}  Let  $\alpha, \beta$ denote   strings in the alphabet  of digits 
\n $0, \ldots, {p-1}$, which are thought of as extended by $0$'s if necessary. Let $[\alpha] = \prod_{i< |\aaa|} x_i^{\alpha_i}$.
 Each element of $L_p$ has a normal form 
\bc $[\aaa] \prod_{s < n} \prod_{r< s } [x_r,x_s]^{\gamma_{r,s}}$ \ec
where $|\aaa| = n$ and $0\le \gamma_{r,s} <p$.  By the usual commutator rules in a group that is nilpotent of class 2,  for $|\aaa| = |\beta|=n$ and central elements $c,d$, 
\[ \big [ [\aaa] c, [\beta] d \big ]= \prod_{r<s\le n} [x_r, x_s]^{\aaa_r \beta_s - \aaa_s \beta_r}.\]
This identity will tacitly be used below.

Assume for a contradiction that $L_p$ has a finite automata presentation with domain a regular set  $D\sub \Sigma^*$, and FA-recognizable operation $\circ$. Denote by $\preceq$ the length-lexicographical ordering on $D$. Note that $\preceq$   can also be recognized by a finite automaton. 

%

Let $C\sub D$ denote the regular set of strings that are in the centre of $L_p$.  
\begin{claim} \label{cl: uS} For each finite set $S \sub D$, there is a $\preceq$-least   string $u=u_S \in  D$  such that 
\bi
\item[(i)] $ \forall r [ r \in S-C \RA  [r,u]\not\in S]$. 
\item[(ii)] $\forall v,w \in S ( [v,u]= [w,u] \to \exists c \in C \, cv = w )$.
\ei
\end{claim}
To see this, let $k$ be so large that only $x_i$ with $i< k$ occur in the normal form of any element of $S$, and let $u=x_k$. 

For (i) note that $r$ contains some $x_i$ with $i<k$, so the   normal form of $[r,u]$ contains $[x_i,x_k]$, while the normal form of an element of $S$ does not contain such commutators.

 For (ii) let $v= [\aaa] c$ with $c$ central. Then the normal form of $[v,u]$ ends in $\prod_{i<k} [x_i,x_k]^{\aaa_i}$, which allows us to recover $\aaa$.  This verifies the claim.

 \medskip
 
 We now inductively define  sequences $\seq{z_n}\sN n$ and $\seq{u_n}\sN n$ in $D$. 
 Let $z_0$ be the string representing the neutral element. 
  Suppose now that  $z_n$ has been defined, and write $V_n$ for $\{v \colon v \preceq z_n\}$.  Let $u_n$ be $u_S$ for $S={V_n}$, where $u_S$ is  defined in the claim above. Next let $z_{n+1}$ be the $\preceq$-least string $z$ such that 
 \begin{equation} \label{eqn:zn+1} \forall v,w \in V_n \, \bigwedge_{i<p} w \circ u_n^i \circ [v,u_n] \preceq z.\end{equation}
Note that $z_{n+1} = G(z_n)$ for a function $G\colon D \to D $ that is first-order definable in the   structure $(D, \circ, \preceq)$. This implies that the graph of $G$ can be recognized by a finite automaton in the usual sense of automatic structures, and hence  $|z_n| = O(n)$ by the pumping lemma.

In the following we write $u_{i,k}$ for $[u_i, u_k]$ where $i\neq k$. 
\begin{claim} \label{cl:subgroup} For each $n$, the subgroup $\la u_0 , \ldots, u_n \ra$ generated by $u_0 , \ldots, u_n$ is contained in $V_{n+1}$. \end{claim}
\n This is checked by induction on $n$. For $n=0$ we have $\la u_0 \ra \sub V_1$ because $w=v=1$ is allowed in  (\ref{eqn:zn+1}). For the inductive step, note that by the normal form (and freeness of $L_p$) each element $y$ of $\la u_0 , \ldots, u_n \ra$ has the form 
\bc $ \prod_{i \le n} u_i^{\aaa_i} \prod_{s \le n} \prod_{r< s } [u_{r,s}]^{\gamma_{r,s}}$. \ec
This can be rewritten as $wu_n^{\aaa_n} [v,u_n]$ where
 \bc $w= \prod_{i <  n}   u_i^{\aaa_i} \prod_{s < n} \prod_{r< s } [u_{r,s}]^{\gamma_{r,s}}$
 and 
 $v= \prod_{k<n} u_k^{\gamma_{k,n}}$. \ec
 By inductive hypothesis $w,v \in V_n$. So the element $y$ is in $V_{n+1}$ by  (\ref{eqn:zn+1}).
This verifies the claim.

In the next claim we view elementary abelian $p$-groups as vector spaces over the field $\mathbb F_p$.
\begin{claim} \label{cl:indep}  \mbox{} \bi \item[(a)] The   elements $C u_i$ are linearly independent in  $G/C$. 
\item[(b)] The elements $u_{i,k}$ are   linearly independent in  $C$.
\ei
\end{claim}
We use induction on a bound $n$ for the indices.
For (a) note that $u_{n+1}^i \not \in V_n C$ for $i<p$, for otherwise $v= u_{n+1}^i c \in V_n$ for some central $c$, and clearly $[v,u_{n+1}]=1$ contradicting the condition (i) in Claim~\ref{cl: uS}. Therefore $C\la u_{n+1}\ra  \cap  C \la u_0, \ldots, u_n \ra =0$. 

For (b), inductively the $u_{i,k}$,  $i<k \le n$ are a basis for  a subspace $T_0\le C$.  The linear map given by $Cw\mapsto  [w,u_{n+1}]$ is 1-1 by (ii) of Claim~\ref{cl: uS}. So, by (a), the $[u_i,u_{n+1}]$ are independent, generating a subspace $T_1$. The condition (i) of Claim~\ref{cl: uS} implies that $T_0 \cap T_1 =0$. This concludes the inductive step and verifies  the claim.

We now obtain our contradiction. Given $n$, by Claim~\ref{cl:subgroup}  we have \bc  $\prod_{i<k \le n} u_{i,k}^{\gamma_i,k} \in V_n$ \ec  for each double sequence $\seq{\gamma_{i,k}}_{i<k \le n}$ of exponents in $[0,p)$. By Claim~\ref{cl:indep}  all these elements are distinct. Since $V_n$ consists of strings of length $O(n)$,  we have $p^{n^2/2}$ distinct strings of length $O(n)$, which is false for large enough~$n$.
\end{proof}

\section{Nies: open questions on classes of  closed subgroups of $S_\infty$}
\label{s:OQ nonarch}
The following started in discussions with A.\  Kechris at Caltech in 2016.  It is well known that the closed subgroups of $\S$ form a Borel space, and there is  a Borel action of $\S$ by conjugation  (see~e.g.\  \cite{Kechris.Nies.etal:18}). This is because being a subgroup is a Borel property in the usual   Effros space of closed subsets of $\S$. A natural class of closed subgroups of $\S$ should be closed under conjugation, and often even  is closed under topological isomorphism (an exception being oligomorphicness). Here we ask which natural classes  of closed subgroups of $\S$ are Borel.  It turns out that even for the simplest classes this can be a hard problem.

 Once a class is known to be Borel, one can  study the  relative complexity of  the topological isomorphism problem for this class with respect to Borel reducibility $\le_B$.  For instance, Kechris et al.\  \cite{Kechris.Nies.etal:18} showed that for the Borel properties of   compactness, local compactness, and Roelcke precompactness,  the  isomorphism relation  is Borel equivalent to  $ \mathrm {GI}$, the  isomorphism of countable graphs. This worked by adapting Mekler's result described below. 

Before proceeding, we import some preliminaries from~\cite{Kechris.Nies.etal:18}.
\subsection*{Effros structure of  a Polish space} Given a Polish space $X$, let $\+ F(X)$ denote  the set of closed subsets of $X$. The \emph{Effros structure} on $X$  is the Borel space consisting of  $\+ F(X)$  together with  the $\sigma$-algebra   generated by the sets \bc $\+ C_U = \{ D \in \+ F(X) \colon D \cap U \neq \ES\}$,  \ec for open $U \sub X$. Clearly it suffices to take all the sets $U$ in a countable basis $\seq{U_i}\sN i$ of $X$.  The  inclusion relation  on $\+ F(X)$ is Borel because for $C, D \in 
\+ F(X)$ we have $C \sub D \lra \fa i \in \NN \,  [ C \cap U_i \neq \ES \to D \cap U_i \neq \ES]$.

\subsection*{Representing elements of the Effros structure of $\S$} For a Polish group $G$, we have a Borel actions $G \curvearrowright \+ F(G)$ by translation and by conjugation.  We will only consider the case that $G= \S$.
  In the following $\sss, \tau, \rho$ will denote injective maps on  initial segments of the integers, that is, on tuples of integers without repetitions. Let $[\sss]$  denote the  set of permutations extending~$\sss$:  \bc $\+ [\sss] = \{ f \in \S \colon \sigma \prec f\}$  \ec (this is often denoted  $\+ N_\sss$ in the literature).  The sets $[\sss]$ form a base for the topology of pointwise convergence of $\S$.  For $f \in \S$ let $f \uhr n$ be the initial segment of $f$ of length $n$. Note that the $[f\uhr n]$ form a basis of neighbourhoods of $f$. 
Given  $\sss, \sigma'$ let $\sigma' \circ \sigma$ be the composition as far as it is defined;  for instance, $(7,4,3,1,0 ) \circ (3,4 ,6 ) = (1,0)$. Similarly, let $\sss^{-1}$ be the inverse of $\sss$ as far as it is defined.

\begin{definition} For $n \ge 0$, let  $\tau_n$ denote  the  function $\tau$ defined on $\{0,\ldots, n\}$ such that $\tau(i) = i$ for each $i \le n$. \end{definition} 
\begin{definition} For $P \in \+ F(\S)$,  by $T_P$ we denote the tree describing $P$ as a closed set in the sense that $[T_P] \cap \S = P$. Note that  $T_P = \{ \sss \colon \, P \in \+ C_{[\sss]}\}$. \end{definition}

\begin{lemma} The closed subgroups of $\S$ form a Borel set  $\+ U(\S)$ in $\+ F(\S)$. \end{lemma} 

\begin{proof} $D \in \+ F(\S)$ is a subgroup iff     the following three conditions hold:

\bi \item $D \in \+ C_{[(0, 1, \ldots, n-1 )]}$  for each $n$
\item $D \in \+ C_{[\sss]}  \to D \in \+ C_{[\sss^{-1}]}$  for all $\sss$
\item $D \in \+ C_{[\sss]}  \cap C_{[\tau]}  \to D \in \+ C_{[\tau \circ \sss]}$ for all     $\sss, \tau$. \ei
It now suffices to observe  that all three conditions are Borel. 
\end{proof}

Note that $\+ U(\S)$ is a standard Borel space.
 The statement of the  lemma  actually holds for each Polish group in place of $\S$.

So much for the preliminaries.  We now  note any known results on the complexity of isomorphism. 
Let $G$ always denote a closed subgroup of $\S$ (another, and cumbersome, but persistent, term is non-Archimedean group).
  By 	`group' we usually mean such a $G$. 
\subsubsection*{The class of all closed subgroups of $\S$} It is not hard to  verify that the isomorphism problem is analytic~\cite{Kechris.Nies.etal:18}.  It is Borel above ($\ge_B$)  graph  isomorphism $\mathrm {GI}$.  Nothing else appears to be  known on its complexity.  The following was asked in 
Kechris et al.\  \cite{Kechris.Nies.etal:18}. 
\begin{question} Is the isomorphism relation between  closed subgroups of $\S$ analytic complete? \end{question}
The isomorphism problem for abelian Polish groups is known to be analytic complete \cite{Ferenczi.etal:09}, but the groups used there are not non-Archimedean (are Archimedean?).

\subsubsection*{Discrete groups} 
Discreteness is  Borel because it is equivalent to saying that the neutral element  is isolated. Note that $G$ is discrete iff $G$ is countable. The isomorphism  relation for  discrete groups is Borel equivalent to graph isomorphism. The upper bound is fairly standard, the lower bound is obtained   by A.\  Mekler's technique \cite{Mekler:81} encoding countable graph isomorphism (for a sufficiently rich class of graphs) into isomorphism of countable nil-2 groups of exponent $p^2$, where $p$ is some fixed odd prime. 

\subsubsection*{Procountable groups} The following are equivalent (see e.g.\ Malicki~\cite[Lemma~1]{Malicki:16}): 

\bi \item[(i)] $G$ is  procountable, i.e., an   inverse limit of a chain of countable groups $G_{n+1} \to G_n$

\item[(ii)]  $G$ is a closed subgroup of a Cartesian product of discrete groups

\item[(iii)] there is nbhd base of the neutral element consisting of open normal subgroups. \ei

It  is also  equivalent to ask that 
\bi
\item[(iv)] $G$ has a compatible  bi-invariant metric (such a metric will be necessarily complete because $G$ is  a Polish group). 
\ei

This class includes  the abelian closed subgroups of $\S$ (see below), and of course the   discrete groups. So graph isomorphism GI can be reduced to its isomorphism problem.

\subsubsection*{Abelian groups} To be abelian is easily seen to be Borel because \bc $G$ is abelian iff $\fa \sss, \tau \in T_G \, [ \sss^{-1} \tau^{-1}\  \sss  \tau \prec \text{id}_\omega]$. \ec (In fact any variety of groups is Borel, by a similar argument.) %
As a nonlocally compact example,  consider $\ZZ^\omega$. Also there is a universal abelian closed subgroup of $\S$.

Each abelian closed subgroup of $\S$ has an invariant metric and hence  is pro-countable by Malicki~\cite[Lemma 2]{Malicki:14}, also  Malicki~\cite[Lemma~1]{Malicki:16}.    Su Gao (On Automorphism Groups of Countable Structures, JSL, 1998) has proved that if $\Aut(M)$ is abelian (or merely solvable)  for a countable structure $M$, then the  Scott sentence of $M$ has no uncountable model.

\subsubsection*{Topologically finitely generated groups}   This property is analytical.  The symmetric group~$\S$ is f.g.\ because the group of permutations with finite support is     dense in $\S$, and is  contained in a 2-generated  group, namely the group generated by the successor function on $\ZZ$ and the transposition~$(0, 1)$.  

\begin{question}Is being topologically finitely generated   Borel? 

\n Given $k \ge 2$, is being $k$-generated   Borel?

  \end{question}

\n We observe that among the compact groups, being $k$-generated is Borel. For in a Borel way we can represent  $G$ as  $\projlim\sN n G_n$ for discrete finite groups $G_n$ (with some unnamed projections $G_{n+1} \to G_n$).  Then  $G$ is $k$-generated iff each $G_n$ is $k$-generated: for the nontrivial implication,  for each $n$ let $\ol a_n$ be  a  $k$-tuple of generators for $G_n$. Now take a converging subsequence of a sequence of pre-images of the $\ol a_n$  in $G^n$. The limit generates $G$.  

%

Being 1-generated (monothetic) is Borel, because as it is well known,  such a group  is either discrete, or an inverse limit of cyclic  groups (e.g.\ $(\ZZ_p, +)$) and hence compact. See e.g.\ Malicki~\cite[Lemma~5]{Malicki:16}.  Hewitt and Ross in their book have a more detailed structure theorem for such groups.

Among the abelian groups, being f.g.\ is Borel because such a group is pro-countable. If it is $k$-generated, it has to be an inverse limit of $k$-generated countable abelian groups. An onto map $\ZZ^s \to \ZZ^s$ is of course a bijection. So  $G$ is $k$-generated iff $G $  is of the form $\ZZ^r \times H$ where $H$ is a product of $k-r$ procyclic groups.  This condition is  Borel.

Outside the abelian, it is easy to provide an example of a complicated pro-countable  2-generated group. In the free group $F(a,b)$ let $v_z = b^{-z} a b^z$ ($z \in \ZZ$). For $k \in \NN^+$  let $N_k \le F(a,b)$ be the normal subgroup generated by commutators $[v_0, v_r], r \ge k$. Then $N_1 > N_2 > \ldots$ and $\bigcap_k N_k = \{1\}$. Let $G$ be the inverse limit of the system $\seq{F(a,b)/N_k}\sN k$ with the natural projections. 

\subsubsection*{Compactly generated groups} One  says that $G$ is compactly generated if there is a compact subset $S$ that  {topologically} generates $G$.   Note that if $G$ is locally compact,  it has a    compact open subgroup $K$ (van Dantzig), so   the compact subset $KS$ generates $G$ algebraically. 

\begin{fact} For locally compact groups $G\le_c \S$, being compactly generated is Borel. \end{fact}
To see this, note that if $G$ is c.g.\ iff it is topologically generated by a compact open set $C$. Such a set $C$ is given as a finite union of sets $[T_G] \cap [\sss]$ for strings on $T_G$, and we can describe arithmetically  whether  a set   $[T_G] \cap [\sss]$ is compact. So we have to express that there is $C$ such that for each $\eta$ on $T_G$, there is a term $t$ and finitely many $\sss_i$ with $[\sss_i] \cap T_G \sub C$ so that $t$ applied to the $\sss_i$ yields an extension of $\eta$. This is arithmetical. 
\begin{question} Is being (topologically) compactly  generated   Borel? \end{question}
\subsubsection*{Oligomorphic groups} To be oligomorphic  means that for each  $n$ there are only finitely many $n$-orbits. Equivalently the orbit structure $M_G$ is $\omega$-categorical. By Coquand's work (elaborated in  Ahlbrandt and Ziegler \cite{Ahlbrandt.Ziegler:86})  oligomorphic groups $G, H$ are isomorphic iff  $M_G$ and $M_H$ are bi-interpretable. 
Nies, Schlicht and Tent~\cite{Nies.Schlicht.etal:19}  have proved that isomorphism of oligomorphic groups is $\le_B E_\infty$, the universal countable Borel equivalence relation. So it is way below graph isomorphism.  In fact it is unknown to be nonsmooth. By Harrington-Kechris-Louveau, nonsmoothness is equivalent to an affirmative answer to the following.

 \begin{question} Is  $E_0 \le_B $  isomorphism of oligomorphic groups? \end{question}

%
%
%
%

\subsubsection*{Amenable   groups (with A.\ Iwanow and B.\ Majcher)}  
Recall that a Polish group $G$ is amenable if   each compact space it acts on has an invariant probability measure. $G$ is extremely amenable if each such action has a fixed point (so the point mass on it is the required probability measure). For discrete groups, this is equivalent to the usual Folner condition.

  Discrete amenable group form
a Borel subset. For,  applying Kuratowski-Ryll-Nardzewski selectors
the Folner condition can be   presented in a Borel form.

Nilpotent groups are amenable. Thus, Mekler's result can be also
applied to the isomorphism relation of discrete amenable groups, making it equivalent to GI. 

Amenability (without the assumption of discreteness) is Borel by   its  characterisation due to Schneider and Thom~\cite{Schneider.Thom:17}.
The description is more less
in the style of Folner.

For more detail, including  extreme amenability, see Section~\ref{s:IvanovMajcher}.

 
 \subsubsection*{Maximal-closed groups}      Fixing  some bijection $\QQ^n \leftrightarrow \omega$, the group AGL($\QQ^n)$ of affine linear transformations can be seen as a closed subgroup of $\S$.    Kaplan and   Simon~\cite{Kaplan.Simon:16} showed that  it is  a maximal closed subgroup (that is also countable).    Agarwal   and Kompatscher~\cite{Agarwal.Kompatscher:18}  have provided   continuum many maximal-closed groups that are not even algebraically isomorphic, using   ``Henson digraphs" that were introduced in  a paper of Henson. 
 
 Clearly being maximal-closed is  $\Pi^1_1$. It is not known to be Borel.
%

 \subsection*{Recursion theoretic view} \label{s:RT} 
The Effros space is insufficient here. We need a more concise way to represent closed subgroups of $\S$.  They are given by trees without dead ends satisfying a certain $\PPI$ condition.

Let $\mathbb T $ be the  tree of all pairs $\la \sss, \sss' \ra$ of the same length $n$ such that
$\sss(i) = k \lra \sss'(k) = i$ for each $i,k < n$. In other words, there is $f \in \S$ such that $\sss \prec f$ and $\sss' \prec f^{-1}$.

If $B $ is a subtree of $\mathbb T$ without dead ends, then for each $\la f , f' \ra \in [B]$, $f$  is a permutation of $\omega$ with inverse $f'$. We can formulate as a $\PPI$ condition on $B$ that $\{f \colon \la f, f^{-1} \ra \in [B]\}$ is closed under inverses and product. 

If $B$ is a computable tree we say that the group given by $[B]$ is computable. 

\begin{question}  Are there   two   compact, computably isomorphic computable   subgroups of $\S$ such that no computable copies are conjugate via a computable permutation of $\S$?. \end{question}

%
%
%
%



\section{Ivanov and Majcher: amenable subgroups of $S(\omega )$} 
\label{s:IvanovMajcher}
%
%


In this post we show that the properties of being 
  amenable and extremely amenable for Polish groups are Borel.

%
\label{intro}

Given a Polish space ${\bf Y}$ let $\mathcal{F} ({\bf Y})$ 
denote the set of closed subsets of~${\bf Y}$. 
The Effros structure on $\mathcal{F} ({\bf Y})$ 
is the Borel space with respect to the $\sigma$-algebra generated by 
the sets 
$$ 
\mathcal{C}_U = \{ D\in \mathcal{F} ({\bf Y}): D\cap U \not=\emptyset \}, 
$$ 
for open $U \subseteq {\bf Y}$. 
For various ${\bf Y}$ this space serves for analysis 
of Borel complexity of families of closed subsets 
(see \cite{Kechris.Nies.etal:18} for some recent results). 
It is convenient to use the fact that there is a sequence 
of Kuratowski-Ryll-Nardzewski selectors  (selectors, in brief)
$s_n : \mathcal{F} ({\bf Y}) \rightarrow {\bf Y}$, $n\in \omega$, 
which are Borel functions such that for every non-empty 
$F\in \mathcal{F} ({\bf Y})$ the set $\{ s_n (F) ; n\in \omega \}$ 
is dense in $F$. 

\bigskip

We consider $S(\omega)$ as a
complete metric space by defining
\[
d(g,h) = \sum \{ 2^{-n} \mid  g(n) \not= h(n)
\mbox { or }g^{-1}(n) \not= h^{-1}(n) \} .
\]
 
Let $S_{<\infty}$ denote the set of all bijections between
finite substes of $\omega$. 
Let 
\[
S^+_{<\infty} = \{ \sigma \in S_{<\infty} \mid \mathsf{dom} [\sigma ]  
\mbox{ is an initial segment of } \omega \} . 
\]
The family  
\[
\{ {\mathcal N}_{\sigma} \mid  \sigma \in S^+_{<\infty} \} 
\]
is a basis of the Polish topology of $S(\omega )$. 

We mention here that the set ${\mathcal U} (S(\omega ))$ 
of all closed subgroups of $S(\omega )$
is a Borel subset of ${\mathcal F} (S(\omega ))$ (see Lemma 2.5 of \cite{Kechris.Nies.etal:18}). 

\bigskip 

Since $S(\omega )$ is a Polish group (in particular the multiplication is continuous) we may extend the set of   selectors $s_n$, $n\in \omega$, by group words of the form $w(\bar{s})$
which define Borel maps ${\mathcal F} (S(\omega )) \rightarrow S(\omega )$
and respectively ${\mathcal U} (S(\omega )) \rightarrow S(\omega )$. 
In particular for any closed $G\le S(\omega )$ all $w(\bar{s}) (G)$ form 
a dense subgroup. 
Below for simplicity we will always assume that already all $s_n (G)$, 
$n\in \omega$, form a dense subgroup of $G$.  

\subsection{Closed subgroups and amenability}  \label{s:closed amenable} 

In this section we apply the description of amenable topological groups
found by F.M. Schneider and A. Thom in \cite{Schneider.Thom:18}  in order to analyse 
amenability for closed subgroups of $S(\omega )$. 

Let $G$ be a topological group, $F_1 ,F_2\subset G$ are finite and $U$ be an identity neighbourhood. 
Let $R_U$ be a binary relation defined as follows: 
$$
R_U=\{ (x,y) \in F_1\times F_2 : yx^{-1} \in U \} . 
$$ 
This relation defines a bipartite graph on $(F_1 , F_2 )$. 
Let 
$$
\mu (F_1 ,F_2 ,U) = |F_1 | - \mathsf{sup} \{ |S| - |N_R (S)| : S\subseteq F_1 \} , 
 $$ 
where $N_R (S) = \{ y\in F_2 : (\exists x\in S) (x,y)\in R_U \}$. 
By Hall's matching theorem this value is the {\em  matching number} of the graph $(F_1 , F_2 , R_U )$. 
Theorem 4.5 of \cite{Schneider.Thom:18} 
gives the following description of amenable topological groups. 

\

{  Let $G$ be a Hausdorff topological group. The following are equivalent. \\ 
(1) $G$ is amenable. \\
(2) For every $\theta \in (0,1)$, every finite subset $E\subseteq G$, and every identity neighbourhood $U$, 
there is a finite non-empty subset $F\subseteq G$ such that 
$$
\forall g\in E (\mu (F,gF,U) \ge \theta |F|). 
$$ 
(3) There exists $\theta \in (0,1)$ such that for every finite subset $E\subseteq G$, and every identity neighbourhood $U$, 
there is a finite non-empty subset $F\subseteq G$ such that 
$$
\forall g\in E (\mu (F,gF,U) \ge \theta |F|). 
$$
}
 
It is worth noting here that when an open neighbourhood $V$ contains $U$ 
the number  $\mu (F,gF,U)$ does not exceed $\mu (F,gF,V)$. 
In particular in the formulation above we may consider neighbourhoods $U$ 
from a fixed base of identity neighbourhoods. 
For example in the case of a closed $G \le S(\omega )$ 
we may take all $U$ in the form of stabilizers  
$V_{[n]} = \{ f\in G : f(i) = i$ for $i<n \}$.  
It is also clear that we can restrict all $\theta$ by rational numbers. 
From now on we work in this case.

\begin{theorem} \label{thmamen} 
The class of all amenable closed subgroups of $S( \omega )$ 
is Borel.  
\end{theorem} 

\begin{proof} 
Since the family of all closed subgroups of $S(\omega )$ is Borel 
it suffices to prove the following claim. 

CLAIM.  For every basic open neighbourhood $U$ of the unity, any rational $\theta \in (0,1)$ and any pair of tuples $\bar{s}$ 
and $\bar{s}'$ of   selectors 
the family of all closed $Z \subseteq S(\omega )$ with the condition 
$$
\forall g\in \bar{s}(Z) (\mu (\bar{s}'(Z),g\bar{s}'(Z),U) \ge \theta |\bar{s}'(Z)|)  
$$
is Borel. 
 
Indeed, let us denote the condition of the claim by 
$F\o (U, \theta ,\bar{s}, \bar{s}' )$. 
Then having Borelness as above we see that 
the (countable) intersection by all $U$, $\theta$ and $\bar{s}$ 
of the families 
$$
\bigcup \{ \{ G \le_c S(\omega ) : G\models F\o (U, \theta ,\bar{s}, \bar{s}' ) \} : 
\bar{s}' \mbox{ is a tuple of } 
$$ 
$$
\mbox{   selectors} \}   
$$ 
is also Borel. 
Note that this family exactly consists of closed subgroups $G$ 
having dense sabgroups satisfying  condition (2) of Schneider-Thom's theorem. 
It is well-known that groups having dense amenable subgroups are amenable. 
In particular we see that the claim above implies the theorem.  

Let us prove the claim. 
For a closed $Z\subseteq  S(\omega )$, $g\in \bar{s}(Z)$ and 
$F= \{ f_1 ,\ldots , f_k \}$ consisting of entries of $\bar{s}'(Z)$ 
to guarantee the inequality $\mu (F,gF,U ) \ge \theta |F|$    
we only need to demand that for every 
$S\subseteq F$ the following inequality holds: 
$$ 
| S| - k + \theta \cdot k \le |N_R (S)| , 
$$
where  $N_R (S)$ is defined with respect to $(F, gF)$ 
and $U$. 
To satisfy this inequality we will use the observation that 
when $S'\subseteq gF$ and $\rho$ is a function 
$S' \rightarrow S$ such that 
$g f (\rho (g f ))^{-1} \in U$ for each $g f \in S'$ 
then $|S'| \le |N_R (S ))|$. 

The following condition formalizes $\mu (F,gF,U ) \ge \theta |F|$: 
$$
\bigwedge_{S\subseteq F} \bigvee \{ 
   \bigwedge_{gf \in S'} (g f (\rho (g f ))^{-1} \in U)
: S' \subseteq gF 
\mbox{ , }
\rho :S' \rightarrow S \mbox{ , }  
$$
$$ 
| S| - k + \theta \cdot k \le |S'| \} . 
$$ 
By the choice of $g$ and $F$ we see that all closed 
$Z\subseteq S(\omega )$ satisfying it form a Borel family. 
\end{proof} 

Let ${\bf U}$ be the Urysohn space. 
By \cite{Uspenskij:90}  every Polish group is realized as a closed subgroup of 
$Iso({\bf U})$. 
Applying the proof given above to ${\+ F}(Iso ({\bf U}))$ 
we obtain the following corollary. 

\begin{corollary} 
The class of all amenable closed subgroups of $Iso({\bf U})$ 
is Borel.  
\end{corollary}

\subsection{Closed subgroups and extreme amenability}   

\label{s:extr amenable}

Let $G$ be a topological group. 
The group $G$ is said to be extremely amenable if every continuous action of $G$ on a non-empty compact Hausdorff space admits a fixed point. 

We begin by fixing a left-invariant metric $d$ inducing the topology of 
$S(\omega )$ (resp. $Iso ({\bf U})$). 
Recall from (\cite{Pestov:05}, Theorem 2.1.11) that $G \le_c S(\omega )$ 
is extremely amenable if and only if the left-translation action of 
$G$ on $(G,d)$ is finitely oscillation stable. 
From (\cite{Pestov:05}, Theorem 1.1.18) and 
(\cite{Melleray.Tsankov:13}, proof of Theorem 3.1) this is equivalent 
to the following condition:
\begin{quote} 
For any $\varepsilon > 0$ and a finite $F\subset G$ there exists a finite 
$K \subseteq G$ such that for any function $c: K \rightarrow  \{ 0,1\}$ 
there exists $i \in \{ 0,1\}$ and $g \in G$ 
such that for any $f \in F$ there exists $k \in c^{-1}(i)$ 
with  $d(gf,k) < \varepsilon$. 
\end{quote} 

We consider the case when $G$ has a countable base of the topology.  
By the definition of extreme amenability if $G$ has a dense subgroup which is extremely amenable, then $G$ is extremely amenable too. 
Now it is easy to see that when $D \subseteq G$ is a countable dense 
subgroup of $G$ then extreme amenability of $G$ is equivalent to  
condition above for the elements taken in $D$.

We now see that when $G\in {\+ F}(S(\omega ))$ 
(resp. ${\+ F}(Iso ({\bf U}))$), extreme amenability of $G$ 
is equivalent to a countable conjunction of the folowing conditions. 
\begin{quote} 
Let $\bar{s}$ be a tuple of   selectors 
and $\varepsilon \in {\bf Q}^+$. 
Then there is a   selectors $\bar{t}$ 
such that for any function $c: \bar{t} \rightarrow  \{ 0,1\}$ 
there exists $i \in \{ 0,1\}$ and a selector $s'$ such that 
for any $s\in \bar{s}$ there exists $k\in c^{-1}(i)$ 
with $d(s'(G)s(G), k(G))<\varepsilon$. 
\end{quote} 
We see that extreme amenability is a Borel property. 

\subsection{Comments} 

1. The argument given in Section~\ref{s:extr amenable} is adapted from 
the proof of Theorem 1.3 in \cite{Melleray.Tsankov:13}. 
Originally \cite{Melleray.Tsankov:13} considers the following situation. 
Let $G$ be a Polish group and $\Gamma$ be a countable group. 
Let us consider the Polish space $Hom(\Gamma ,G)$ of all 
homomrphisms from $\Gamma$ to $G$. 
By Theorem 3.1 in \cite{Melleray.Tsankov:13} 
the subset of all $\pi \in Hom(\Gamma ,G)$ such that  
$\overline{\pi (\Gamma)}$ 
is extremely amenable is  a $G_{\delta}$ subset of $Hom(\Gamma ,G)$.
By Corollary 18 of \cite{Kaichouh:15} the set of all representations 
from $Hom(\Gamma, G)$ whose image is an amenable subgroup of $G$ is 
also $G_{\delta}$ in $Hom(\Gamma ,G)$. \\ 
2. Let ${\+ G}_n$ be the space of all $n$-generated (discrete) 
groups with distinguished $n$-tuples of generators $(G, \bar{g})$ 
(so called {\em marked groups}). 
This is a compact space under so called {\em Grigorchuk topology}. 
In papers \cite{Benli.Kaya:19} and \cite{Wesolek.Williams:17} descriptive complexity 
in ${\+ G}_n$ of some versions of amenability is considered. 
The authors of \cite{Benli.Kaya:19} show that amenability is ${\bf \Pi}^0_2$. 
They ask if it is ${\bf \Pi}^0_2$-complete.  

A group $G$ is called {\em elementarily amenable} if it is in the smallest class of groups which contains all abelian and finite ones and is closed 
under quotients, subgroups, extensions corresponding to 
exact sequences $1 \rightarrow K \rightarrow G \rightarrow H\rightarrow 1$ and directed unions. 
It is proved in \cite{Wesolek.Williams:17} that elementary amenability is coanalytic 
and non-Borel. \\ 
3. Let us fix an indexation of all computably enumerable groups on 
$\omega$ (i.e. computably presented groups). 
Under this indexation computable groups correspond to 
groups with decidable word problem. 
It is easy to see that F\o lner's condition of amenability 
(or the Schneider-Thom's condition of Section~\ref{s:closed amenable}  in the case of 
discrete groups) define a $\Pi^0_2$ subset of indices.   
On the other hand applying Theorem 3 of \cite{Bilanovic.etal:20} it is easy to see that this property is $\Pi^0_2$-hard (it is a Markov property). 
Similarly one easily obtains that extreme amenability is $\Pi^0_2$-complete.

  E-mail:  Aleksander.Iwanow\@polsl.pl


\newcommand{\mor}[3]{{\,  }_#2 #1_#3}
\newcommand{\morn}[3]{#1\colon #2 \to #3}
\newcommand{\Op}{\text{Op}}
\newcommand{\Inv}{\text{Inv}}
\newcommand{\rdom} {\mathbf r}
\newcommand{\ldom}{\mathbf l}

\section{Nies: Stone-type duality for  totally disconnected locally compact   groups}
\n  In this post all topological groups  will be   Polish, and they all have a basis of neighborhoods of $1$ consisting of open subgroups. As is well-known, such a group is  topologically isomorphic to a closed subgroup of the symmetric group on $\NN$, denoted~$\S$. A~homeomorphic embedding into $\S$ is obtained for instance  by letting the group act by left translation on the   left  cosets of open subgroups in that basis of neighborhoods of $1$.
 
 Nies, Schlicht and Tent~\cite{Nies.Schlicht.etal:19} developed the notion of coarse groups for closed subgroups of $\S$, which    first appeared in \cite{Kechris.Nies.etal:18}. The idea is to do algebra with approximations of elements,  rather than with the elements themselves. The  approximations are all the cosets of open subgroups (left or right cosets, this makes the same class). Open cosets form the domain of the coarse group, and the structure is equipped with the ternary relation $AB \sub C$. The authors in~\cite{Nies.Schlicht.etal:19} apply  the notion primarily for  the   class of    oligomorphic groups, but also the profinite groups. (Ivanov has pointed out that in that case a closely related structure was studied much earlier by Chatzidakis~\cite{Chatzidakis:98}.) 
 
 Here we give a different  approach to  coarse groups, which is particularly intended   for  the setting of   totally disconnected locally compact (t.d.l.c.) groups. General references for t.d.l.c.\ groups include Willis~\cite{Willis:94,Willis:17}. 
 
 
  In   \cite{Nies.Schlicht.etal:19} all open cosets of a topological group $G$ were considered, but  the analysis was  restricted to classes of groups $G$ which have   only countably many open subgroups. This is e.g.\ the case for   Roelcke precompact groups (for each open subgroup $U$ there is a  finite set $F$ such that $UFU= G$). Such groups are  in a sense opposite to the t.d.l.c.\ groups: the   intersection of those two `large" classes  consists merely of  the profinite groups.  However, a superclass of both has also been studied: locally Roelcke precompact groups.
  
  The coarse group $\+ M(G)$ of a t.d.l.c.\ group $G$ consists of the \emph{compact} open cosets of $G$.

  \subsection{Inductive groupoids, and inverse semigroups}
  A category is small if the objects form a set (rather than a proper class).
  Recall that a groupoid is a small category such that each morphism $A$  has an inverse, denoted $A^{-1}$.  A partially  ordered  groupoid is  a groupoid with a partial order $\sq$ on the set of morphisms (and therefore also on the objects, which are identified with their identity morphisms)  where    the functional and the order structure are compatible.
  An \emph{inductive groupoid} is a  partially ordered groupoid  such that the partial order  $\sq$ restricted to   the set of  neutral elements is a semilattice. See  Lawson~\cite[Section 4.1]{Lawson:98}. 
    
 A semigroup is  called \emph{regular} if for  each $a$ there is $b$, called the  inverse of $a$,  such that $aba=a $ and $bab= b$.  Inductive groupoids closely correspond to  \emph{inverse semigroups}. These are regular semigroups where the idempotents (elements $e$ such that $ee= e$) commute.   In particular, the (large) categories of  inductive groupoids and of inverse semigroups are isomorphic. 
 
 For instance, given an  inductive groupoid, to define the semigroup operation, simply let $AB = (A \mid V) (V \mid B)$, where $V$ is the meet of the right domain of $A$  and the left domain of $B$, and $\mid$ denotes restriction, given by axiom $A2(a)$ below.  See  Lawson~\cite{Lawson:98} for detail.
 
  The  representation theorem     due to Wagner and Preston (both 1954, independently) realizes  every inverse semigroup $S$ as an inverse semigroup of partial bijections on $S$. An element $a \in S$ becomes the  partial bijection $\tau_a \colon a^{-1} S \to S$ given by $t \mapsto at$. Clearly $\tau_b \circ \tau_a= \tau_{ba}$. If $S$ is a group this is just the   left Cayley representation. See  Lawson~\cite[Section 4.1 and 4.5]{Lawson:98}.

 \subsection{The coarse groupoid  of a topological group} Given a topological group for which the open subgroups form a nbhd basis of $1$,  an inductive groupoid is obtained  as follows. 
   \bi \item  Objects correspond to open subgroups. In the abstract setting  they will be  called $^*$subgroups. We use letters $U,V,W$ for them.  
  \item Morphisms correspond to open cosets.   Abstractly they are called $^*$cosets. We use letters $A,\ldots, E$ to denote them. 
    $A\colon U \to V$ means that $A$ is a right coset of $U$ and a left coset of $V$. In brief we often write ${}_U A_V$ for this. The usual  notation in the theory of groupoids is $U= \mathbf d(A)$ (for domain) and  $V = \mathbf r(A)$ (for range).
  \ei
  We will  treat   coarse groupoids axiomatically. We begin with the following.

  \n {\bf Notation and conventions.}    An object  $U$ will be  identified with the neutral  morphism $1_U$. So there are only morphisms, and  objects merely form  a convenient manner of speaking.     We write $RC(U)$ and $LC(U)$ for the sets of    right, resp.\  left $^*$cosets of $U$.   In formulas we also write $_U A$ to mean  that $A \in RC(A)$, and $A_U$ to mean that $A\in LC(U)$.     
  \medskip 
  
  \

 We axiomatically require  the usual properties defining  groupoids and partial orders. For ease of language we adjoin a least element $0$ to the partial order. We  require that in  the partial order $\sq$ on the objects (i.e., the $^*$subgroups),  any two elements $U,V$ have an infimum, denoted $U \wedge V$. By $A\perp B$ denote that $A \wedge B=0$  are incompatible.  If $M= \+ M(G)$ then $0$ is interpreted as the empty set.

 We have the following axioms connecting  the  groupoid  and partial order. (Keep in mind that we identify $U$ and $1_U$.)
  
  \medskip
  
 \begin{axiom}
  \begin{enumerate}
  
    \item[(A1)] If $A \sq B$ then $A^{-1} \sq B^{-1}$.
    
  \item[(A2)] Let $U \sq V$.

    \n ($\DA$)  If $_V B$ then $A \sq B$ for some $_U A$. 

\n   ($\uparrow$) If ${}_U A$ then $A \sq B$ for some $_V B$.

 \item[(A3)] if $A  B$ and $A'   B'$ are defined and 
  $A \sq A', B \sq B'$, then $A  B \sq A'   B'$.

  \item[(A4)] If $_UA$ and $_VB$ and $U \sq V$, then either $A\sq B$ or $A\perp B$.

  \item[(A5)] If $A \not \sq B$ then there is $C \sq A$ such that $C \perp  B$.
  
  \end{enumerate}
  
  \end{axiom}
\n   Remarks: 
  
  \n Note that Axioms (A1), (A2$\DA$) and (A3) are the usual axioms of ordered groupoids, OG1, OG3 and OG2 respectively  in Lawson~\cite[Section 4.1]{Lawson:98}, only the notation there is a bit different. 
  
  We have  $A_U$ iff $_U A^{-1}$ by the definitions, which implies that  the axioms mentioning right $^*$cosets  also holds for left $^*$cosets. See e.g.\  \cite[Section 4.1, Prop 3(6)]{Lawson:98} for a proof of the left coset version of (A2$\DA$) which Lawson calls (OG3$^*$).
  
 Axiom  (A2$\UA$) doesn't seem to occur   in the ordered groupoids literature. 
 Axiom  (A4) is special to the applications to topological groups we have in mind here.  It implies that different right  $^*$cosets of the same $^*$subgroup are disjoint.
  Axiom  (A5)   essentially says that the topology  is Hausdorff.  
  
  \subsection*{The axioms are satisfied for structures of  suitable open cosets}
  
  In the following,  $G$ is a topological group as above with countably many open subgroups, or $G$ is a t.d.l.c.\ group.
  Let $\+ M(G)$ denote the coarse groupoid:  the $^*$subgroups are the open subgroups in the former case, and the compact open subgroups in the t.d.l.c.\ case. The morphisms are the (compact) open cosets. We have $\morn A U V$ if $A$ is a right coset of $U$ and a left coset of $V$. Recall that in brief we write $ _U A_V$ for this.   It is easily seen that the axioms above hold.  To show that $\+M(G)$  (with 0 interpreted as the empty set)  is a lower semilattice,  suppose that   $x \in aU \cap bV$ for subgroups $U,V$, then $xU = aU$ and $xV = bV$. Let $W = U \cap V$. Then $aU \cap bV= x W$.   Claim~\ref{cl: cap_coarse} below shows that this argument works in the  general axiomatic setting.
  
  Note that  $\ex \morn AUV$ iff $U$ and $V$ are conjugate in $G$. In this case, there is $a\in G$ such that $Ua= A= aV$.

 \subsection*{Some   consequences of the axioms} \
 
 \n
First we check that the ordering relation  of morphisms carries over to   their left and right domains.
\begin{claim} \label{cl:incl subgr} \ 

\n Suppose $\morn  A  {U_0}  {U_1}$, $\morn  B  {V_0}  {V_1}$ and $A \sq B$. Then $U_i \sq V_i$ for $i= 0,1$. \end{claim}
\n To verify this: by (A1) we have $A^{-1}\sq B^{-1}$. Then by (A3) and identifying $U$ with $1_U$, we have $U_0= A A^{-1} \sq B B^{-1}=V_0$. Similarly, we've got $U_1 \sq V_1$.

\begin{claim} \label{axiom existence of full filters} 
\

\n For each   $A\in M$ and each $^*$subgroup $U$, there are  a $^*$subgroup $V\sqsubseteq U$ 
and a left $^*$coset $B$ of $V$ 
such that $B\sqsubseteq A$. A similar fact holds for right $^*$cosets. 
\end{claim} 
 \n To see this, suppose that   $A $ is left $^*$coset of $W$. Let $V= W \wedge U$. By Axiom~(A2$\DA$) there is $B\sqsubseteq A$ such that $B$ is left $^*$coset of $V$, as required.

 The following  holds more generally in ordered  groupoids.   
 \begin{claim}[\cite{Lawson:98}, Section 4.1, Prop 3(5)] \
 
 \n  If $C \sq AB$ then there are $A' \sq A$ and $B' \sq B$ such that $C= A'B'$. \end{claim}

 
 Next we show that each left $^*$coset of a $^*$subgroup $V$ is given by the left $^*$cosets of a   $^*$subgroup $U$  it contains. (In a sense it is the ``union" of these cosets.)
 \begin{claim} \
 
 \n  Suppose $U \sq V$. If $B_V \neq C_V$ then there is $A_U \sq B$ such that $A\perp C$. \end{claim}
 \n To verify this, we may suppose that $B \not \sq C$. By Axiom (A5) there a $^*$subgroup $W$ and   $D_W \sq B$ such that $D \perp C$. Let $U'= W \wedge U$. Let $E_{U'} \sq D$ by (A2$\DA$). There is $A_U \sqsupseteq E$ by (A2$\UA$). Since $A\perp B$ fails (because of $E$) we have $A \sq B$ by (A4). However,   $A \sq C$ would imply $E \sq C $ and hence contradict $D \perp C$. So $A \perp C$ by (A4) again.
 
 \begin{claim} \label{cl: cap_coarse} Suppose $A_U   \wedge B_V\neq  0$. Let $W= U \cap V$. Then $A \wedge B$ is the unique left $^*$coset of $W$ contained in $A$ and $B$. \end{claim}
 
 \n If  $C_{W'}\sq A,B$, then $W' \sq W$ by Claim~\ref{cl:incl subgr} and definition of $W$. So by (A2$\UA$) there is $D_W$ such that $C \sq D$. Letting $C= A \wedge B$, we see that $A\wedge B$ is a left $^*$coset of $W$. If any left $^*$coset of $W$ is contained in $A,B$ it equals $C$ by (A4).
 
\subsection*{Normal $^*$subgroups}

Recall that we  write $LC(U)$ and $RC(U)$ for the sets of   left, resp.\ right,  $^*$cosets of $U$.
We say that $^*$subgroups $U,V$ are \emph{conjugate} if $\ex A  \colon U \to V$, or in other words, $RC(U) \cap LC(V) \neq \ES$. In $\+M(G)$ this replicates the usual meaning of conjugacy. For the slightly nontrivial direction,  if $A = Ua = bV$ then $a^{-1} Ua= a^{-1} b V$. This is a subgroup, so $a^{-1} b \in V$, and hence $U^a= V$. 
The axioms   of   groupoids   imply  that conjugacy is an equivalence relation. 

Normal $^*$subgroups $V$ are the ones only conjugate to themselves: for each $B \in RC(V)$ we have $B^{-1}V B= V$. This is equivalent to $VB = BV$ for each such $B$, or equivalently  defined   by the condition $LC(V)= RC(V)$. In category language,  all morphisms with left domain $V$ also have right domain $V$, and vice versa. So there is a natural group operation on $RC(V)$. 

A rather trivial fact from group theory becomes more demanding in the axiomatic setting of coarse groupoids.  
 \begin{prop} \label{prop:normal subgroup}  Suppose that $U$ is a $^*$subgroup such that $RC(U)$,   or equivalently  $LC(U)$,  is finite. Then there is a normal $^*$subgroup $N \sq U$.  \end{prop} 
 
In usual topological group theory, the argument is as follows.  Since $U$ is a subgroup of $G$ of finite index, the conjugacy class of $U$ is finite. Let $N = \bigcap_{g\in G}  U^g$. This is a finite intersection and hence defines an open subgroup, and   $N^h = \bigcup_g U^{gh}= N$ for each $N$. So $N$ is normal.

 \begin{proof} Let $\+ D$ be the ``conjugacy class" of $U$, namely
 
  \bc $\+ D= \{ B^{-1}U B \colon B \in RC(U)\}$. \ec
  The  hypothesis implies that  $\+ D$ is finite, so let $N$ be the meet of all its members. Then $N \sq U$. We show that $N$ is normal as required. 
  
  Given  $C\in RC(N)$, we will   show that $C \in LC(N)$.  
  We define a bijection $f_C \colon \+ D \to \+ D$ by 
  \bc $f_C(W)= D^{-1} W D$ where $D \in RC(W)$ and $C \sq D$; \ec
  note that $D$ exists and is unique because $N \sq W$, so $f_C$ is well-defined. 
  
 To verify that  $f_C$ is bijection,  since $\+ D$ is finite it  suffices to show that $f_C$ is 1-1. Suppose $f_C(W_0)= f_C(W_1)=: V$. Then $D_0^{-1} W_0D_0= D_1^{-1} W_1D_1 = V$ where $C \sq D_0, D_1$ and $D_i \in RC(W_i)$. Since   $D_0, D_1  \in LC(V)$, and $D_0, D_1$ are not disjoint, this implies $D_0 = D_1$ and hence $W_0= W_1$. 
  
We have $N' := C^{-1} N C \sq f_C(W)$ for each $W \in \+ D$ using (A1) and (A3). So,  since $f_C$ is onto, $N' \sq N$. 
  
 Since  $C \in LC(N')$ and $N' \sq N$,  there is $D \in LC(N) $ such that $C \sq D$ by (A2$\UA$). 
  Then $C':= D^{-1} \in RC(N)$, so $C' \sq D' $ for some $D' \in LC(N)$ by the argument   above. Then $D \sq E:= (D')^{-1} \in RC(N)$ by (A1). So $C\sq E$ and both are in $RC(N)$. Hence $C=D= E$ by (A4).  This shows that  $LC(N)\sub RC(N)$, hence also $RC(N) \sub LC(N)$ by taking inverses. So $RC(N)= LC(N)$   as required. 
  \end{proof} 
     \subsection*{The   filter group associated with a coarse groupoid}
   We've seen how to turn a topological group into  a coarse groupoid. Now we go the opposite way. This is adapted from~\cite{Nies.Schlicht.etal:19}. An alternative, possibly easier way to do this is to take the topological group of ``left automorphisms" of $M$, as detailed in Prop.\ \ref{prop: left autom} below.

   Let $M$ be a coarse groupoid. 
   A filter on a p.o.\ is a proper subset that is downward directed and upward closed. For  $(M, \sq)$, a filter is thus a subset $x$ that is upward closed, and $A\wedge B$ exists and is in $x$,  for any $A,B \in x$.
   
   \begin{definition} \label{df:fullfilter} A~\emph{full filter} is a  filter $x$ on the partial order $(M, \sq)$ such that for each $^*$subgroup $U\in M$, there is a left $^*$coset and a right $^*$coset in~$x$. Note that these  $^*$coset are unique by (A4).
 $\+ F(M)$ denotes the set of full filters. We use variables  $x,y,z$ for full filters.  \end{definition}
 
 \begin{claim} \label{cl:exist filter}For each $A$ there is a full filter $x$ such that $A \in x$. \end{claim}
\n  This follows by iterated applications of  Claim~\ref{axiom existence of full filters}. 
 
 We begin with the topology on $\+ F(M)$.
\begin{definition}[Topology on the set of  full filters]  \label{def:topo}  \  \\  As in  \cite{Nies.Schlicht.etal:19} we define a  topology on $\+ F(M)$ by declaring  as subbasic        the  open sets 
 
\bc $\hat{A}=\{x\in \+ F(M)\colon A\in x\}$ \ec where  $A\in M$.  These sets    form a base since  filters are directed.  \end{definition}
Suppose $M$ is countable. The following improves  \cite[Prop 2.5]{Nies.Schlicht.etal:19} that $\+ F(M)$ a totally disconnected Polish space.   
\begin{prop}  There is a homeomorphic  embedding taking $\+ F(M)$ to a closed subset of Baire space. \end{prop}     
\begin{proof}
Let $\seq{U_n} \sN n$ be a descending sequence of $^*$subgroups that is cofinal (every $^*$subgroups contains an $U_n$). Fix a bijection $ f\colon \NN \to M$.   
We define an  injection $\Delta$ from $\+ F(M)$ into   Baire space ${}^\omega \omega$.

  Suppose that $x\in \+ F(M)$. Let $\Delta(x)(0)$ be the  left $^*$coset of $U_0$ in $x$. 
  
  Suppose  $\Delta(x)(2n)$  is defined  and a left $^*$coset of some $U_r$.
  Let $\Delta(x)(2n+1)$ be the $k$ such that $A=f(k) \in x$, and $A $ is a right $^*$coset of $U_m$ where $m > r$ is chosen least possible.  This exists by Claim~\ref{axiom existence of full filters}, since the sequence  $\seq{U_n}$ is cofinal, and (A2).

Similarly,   suppose  $\Delta(x)(2n+1)$  is defined  and a right  $^*$coset of some $U_r$.
  Let $\Delta(x)(2n+2)$ be the $k$ such that $A=f(k) \in x$, and $A $ is a left  $^*$coset of $U_m$ where $m > r$ is chosen least possible.  
  
  By the axioms,  $\Delta$ is injective because $x$ is the filter generated by $\Delta(x)$. One checks that it is a homeomorphism because full filters correspond to the paths on the subtree of the strings given by the  possible next choices at each step.
\end{proof}

Next we define the group operations on    $\+ F(M)$. For filters $x,y$ we let 
 \begin{eqnarray*} x^{-1}&=& \{A^{-1} \colon \, A\in x\}\\
  x  y &= &\{ C \colon \ex A, B  [ A   B\sq C]\}\end{eqnarray*}
  Here   writing $A   B$   implies  that it  is defined, i.e. $\ex U A \ _U B$.
  \begin{claim} If $x,y$ are full filters, then $ x^{-1}$ and $z=x  y$   are full filters. \end{claim}
That $x^{-1}$ is a full filter  is straightforward. 
  For the second statement, clearly $z$ is upwards closed. We verify  that $z$ is downwards directed. 
  
  Always let $i=0,1$.  Suppose $C_i \in z$. Then there are $A_i \in x$, $B_i\in y$ and $U_i$ such that  $ {A_i \ }_{U_i} B_i$   and $A_i   B_i \sq C_i$. 
  
  Let $U = U_0 \wedge U_1$.  By definition of full filters there are $A\in x$ and $B\in y $ such that $A \ _U \, B$. Then  $C=A  B \in z$. By  (A4) and since filters are downwards directed, we have $A\sq A_i$ and $B \sq B_i$. So   by  (A3) we have $C \sq C_i$   as required.
  
  The neutral element $e$  is the full filter consisting of all the $^*$subgroups. It is clear from the groupoid axioms that  $(\+ F(M), \cdot)$ is a group with this neutral element, and the inverse operation above.
  
  That leaves continuity of the group operations. First,  
  as in \cite[Claim 2.11]{Nies.Schlicht.etal:19} we need to check that the transfer from  formal to semantic concepts works between $^*$cosets $A$ and the corresponding  actual open cosets $\hat A$. Note that we are NOT claiming that these are the only open cosets; this is not true unless we require further axioms special to the particular class of groups we are interested in.

  \begin{claim} \label{subset is correct}  Let  $A,B,C \in M$.  
\bi

\item[(a)]  $A\sqsubseteq B \Longleftrightarrow \hat{A}\subseteq \hat{B}$. 
\item[(b)] $\hat {B^{-1}} = (\hat{B})^{-1}$. 

 \item[(c)] If $A\cdot B$ is defined then 
$\hat {A \cdot B} =  \hat A \hat B$.
\item[(d)]  $\hat U$ is a subgroup of $\+ F(M)$.
\item[(e)]  $A\in LC(U) \Longleftrightarrow\hat A $ is a left coset of $\hat U$. 

\n Similarly for right cosets.

 \ei
\end{claim} 
\begin{proof} 

(a) The implication $\RA$ is upward closure of full filters. For the implication $\LA$, suppose that $A \not \sq B$. By (A5) there is $C \sq A$ such that $C \perp B$.  By Claim~\ref{cl:exist filter} let  $x$ be a full filter such that $C \in x$. Then $A \in x$ and $B \not \in x$. 
(b) is immediate. For 
(c),  $\supseteq$ is by definition, and  $\sub$ follows from Claim~\ref{cl:exist filter}.

\n (d) is immediate using $UU=U$.

\n (e) We follow \cite{Nies.Schlicht.etal:19}:

\n $\RA$: take any $x\in \hat{A}$. 
We  show that $x\hat{U}=\hat{A}$.

For $x\hat{U}\subseteq \hat{A}$, let $y\in \hat{U}$. Since $A\in LC(U)$, we have $AU\sqsubseteq A$. So $x\cdot y\in \hat{A}\hat{U}\subseteq\hat{A}$ by (c).

For  $\hat{A}\subseteq x\hat{U}$, let $y\in \hat{A}$. To show  that $y\in x\hat{U}$, or equivalently $x^{-1} y\in \hat{U}$,  note that we have $x^{-1}y\in \hat{A}^{-1}\hat{A}=\hat{A^{-1}}\hat{A} \subseteq \hat{U}$ by (b) and (c). 

\n $\LA$: Suppose $\hat A = x \hat V$. There is $B \in x$ such that $B \in LC(V)$. By the forward implication, $\hat B$ is a left coset of $\hat V$. Also $x \in \hat A \cap \hat B$, so $A\perp B$ fails.  Since $A, B \in LC(V)$ this implies $A=B$ by (A4). 

The case of right cosets follows by taking inverses.
\end{proof}   
Another transfer  fact will be useful. 
 \begin{claim} The map $x, y \to x\cdot y^{-1}$ is continuous on $\+ F(M)$. \end{claim}
\n This follows since the sets of the form $\hat D$ form a basis. If $x  y^{-1} \in \hat D$, i.e., $D \in  x  y^{-1}$,  then by definition there are $A \in x$ and $B\in y $ such that $A B^{-1} \sq D$. Now,  by  Claim~\ref{subset is correct},   $\hat A \hat B^{-1} \sub \hat D$ as required. 
  
  \begin{claim} \label{correct form of left cosets}  \
  
  \n
For any left coset $x\hat{V}$ in $\+ F(M)$, there is $A \in LC(V)$  such that  $x\hat{V}=\hat{A}$. 
\end{claim} 
\begin{proof} 
Since $x$ is a full filter, there is some left $^*$coset $A$ of $V$ in $x$. 
We claim that $x\hat{V}= \hat{A}$. We have $x\hat{V}\subseteq \hat{A}\hat{V}=\hat{V}$, since $A\in x$ and $\hat{A}$ is a left coset of $\hat{V}$ by Claim \ref{subset is correct}. To see that $\hat{A}\subseteq x\hat{V}$, let $y\in \hat{A}$. Since $x,y\in \hat{A}$, $x^{-1}y\in \hat{A}^{-1} \hat A=\hat{A^{-1}} \hat A\sqsubseteq \hat{V}$ by Claim \ref{subset is correct}. Thus $y  \in x\hat{V}$. 
\end{proof}

 By definition of the topology,  the open subgroups of $\+ F(M)$ form a nbhd base of $1$. So if $M$ is countable,  $\+ F(M)$ is a non-Archimedean Polish group. 
  
The operation $\+ F$ recovers a topological group from its coset structure when that is countable. It also works in the t.d.l.c.\ setting where $\+M (G)$ denotes the compact open cosets.
   \begin{prop}[cf.\ ~\cite{Kechris.Nies.etal:18}, after Claim 3.6, and \cite{Nies.Schlicht.etal:19},   Prop 2.13]  \label{fact:standard} \

\n Suppose that $G $ is a closed subgroup of $ \S$ such that  $\+ M(G)$ is countable. There is a natural group homeomorphism  \bc $\Phi: G \cong \+ F( \+ M(G))$ given  by   $g \mapsto \{ A \colon A \ni g\}$, \ec with inverse   given by $x \mapsto g$ where $ \bigcap x = \{g\}$.    \end{prop}
The inverse map simply sends a full filter $x$ to the point it converges to. Note that $x$ isn't really a filter in the sense of topology, only on certain open sets, but that suffices for the convergence notion.

\begin{example} 
For   an instructive  example of a coarse groupoid,    consider the oligomorphic group~$G= \Aut(\mathbb Q, <)$. The open subgroups of $G$ are the stabilizers of finite sets. If $U,V$ are stabilizers of sets of the same finite cardinality, there is a unique morphism $A \colon U \to V$ in the sense above, corresponding to the order-preserving bijection between the two sets. The coarse groupoid for $\Aut(\mathbb Q, <)$ is   canonically isomorphic to the groupoid of finite order-preserving maps on $\mathbb Q$, with the partial order being   reverse extension.  For compatible maps $A,B$, the meet $A \wedge B$ is the union of those maps.

A filter $x$   corresponds  to an arbitrary  order-preserving map $\psi$ on $\mathbb Q$. The filter $x$ contains a right coset of each open subgroup iff $\psi$ is total, and a left coset of each open subgroup iff $\psi $ is onto. So the set of full filters corresponds to $\Aut (\mathbb Q)$ as expected. (Incidentally, this example shows that in Definition~\ref{df:fullfilter}  we  need both types of cosets, and that not every maximal filter is full.) \end{example}

\subsection*{The filter group as an automorphism group}
Let $M$ be a coarse groupoid. By $M_{\text{left}}$ we will denote the structure with domain $M$ and the operations $\wedge$ and $(r_B)_{B\in M}$ where  $A r_B= AB$ in  case $\rdom(A)= \ldom (B)$, and $A r_B=0$ otherwise. We show that  the left action of $\+ F(M)$  on $M$ corresponds to the automorphisms of this ``rewrite" of $M$. (This is simlar to showing that a group is isomorphic to the automorphism group of a Cayley graph given by a generating set, with edge relations labelled according to the generators.)  Note that for each automorphism $p$ of  $M_{\text{left}}$, and each $A$, we have \[\rdom (p(A))= \rdom(A).\]
This is because where $U = \rdom(A)$, we have $p(A)U=  p(AU)= p(A)$. Also, note that $p$ is determined by  its restriction  to the $^*$subgroups, because for each right $^*$coset $B$ of  a $^*$subgroup $U$ we have $p(B) = p(U) B$. 
\begin{prop} \label{prop: left autom} $\+ F(M) $ is topologically isomorphic to $\Aut(M_{\text{left}})$ via a canonical isomorphism $\Theta$.  \end{prop}

\begin{proof} For $x \in \+ F(M)$, the left action $_UB \mapsto A= x\cdot B$ is given by $A = CB$ where $C_U \in x$. The isomorphism $\Theta \colon \+ F(M) \to \Aut(M_{\text{left}})$   maps $x$ to its left action: \bc $\Theta(x)(A) =  x\cdot A$. \ec Clearly $\Theta(x)\in \Aut(M_{\text{left}})$, and $\Theta$ preserves the group operations. 

Let $s_M \in \+ F(M) $ denote the full filter of $^*$subgroups, which is the neutral element of $\+ F(M)$. We claim that the inverse $\Phi$ of $\Theta$ is given by     \bc  $\Phi(p)= p(s_M)$. \ec 

Clearly  $x= \Phi(p)$ is a filter. To show that $x$ is a \emph{full} filter, let $U$ be a $^*$subgroup in $M$. Since $p$ is an automorphism, firstly,  we have $p(U)U = p(U)$, so $p(U) \in LC(U)$ and $p(U) \in x$. Secondly,  there is $B$ such that  $p(B^{-1})=U$.   Then $p(B^{-1}B)= UB=B$. So $B\in RC(U)$. Now $V=B^{-1}B= \rdom(B)$ is a $^*$subgroup. Since $p(V)=B$ we have $B \in x$. 

We verify that $\Theta, \Phi$  are inverses of each others.
\n $\Phi(\Theta(x))= x$ because

 $A_U \in x \lra xU = A \lra \Theta(x)(U)= A$. 

\n $\Theta(\Phi(p) )= p$ because 

$p(U) = A \lra A \in \Phi(p) \lra \Phi(p)U= A \lra \Theta(\Phi(p))(U)= A$.

To show $\Theta$ and $ \Phi$ are continuous at $1$, note that if  $p= \Theta(x)$, then $p(U)= U$ is equivalent to  $x \in \hat U$.  

\end{proof}

\subsection*{Profinite groups, and $^*$compact coarse groupoids} 
\
As mentioned, Chatzidakis~\cite{Chatzidakis:98} carried out the first research related to the application of coarse groupoids to profinite groups. In her version the coarse groupoid was restricted to   normal open $^*$subgroups, which suffices in that case. (Thanks to the Ivanovs for pointing this out.)

  We say that  a coarse groupoid $M$ is \emph{$^*$compact} if $\forall U \, [LC(U) \text{ is finite}]$. This is, of course,  equivalent to requiring that $\forall U \, [RC(U) \text{ is finite}]$.  Clearly $\+ M(G)$ for profinite $G$ has this property. 
By Prop.~\ref{prop:normal subgroup},  
$^*$compactness implies that   each $^*$subgroup $U$ contains a normal $^*$subgroup $V$. (This was required separately in  \cite{Nies.Schlicht.etal:19}.)
 \begin{prop}  \label{prop:profinite} Let $M$ be a coarse groupoid.  Then 

    $M$ is $^*$compact $\LR$ $\+ F(M)$ is compact. \end{prop}

\begin{proof}  $\LA$: Given $U \in M$, by Claim~\ref{subset is correct}(d) $\hat U$ is open in $\+ F(M)$. So it has finite index. By (e) of the same claim, this implies that $LC(U)$ is finite. 

\n  $\RA$:  Using   Prop.~\ref{prop:normal subgroup},  let $\seq{N_k}\sN  k$ be a descending chain of normal $^*$subgroups such that $\forall U \exists k \, [ N_k \sqsubseteq U]$. Let $G_k$ be the group  induced by $M$ on $LC(N_k)$. We define an onto map $p_k \colon G_{k+1} \to G_k$ as follows: given $A\in LC(N_{k+1})$,   using (A2$\UA$) let  $p_k(A)= B$ where $A \sqsubseteq B\in LC(N_k)$. Each $p_k$ is a homomorphism by Axioms (A1, A2).

Let $G$ be the inverse limit:  $ G=\projlim_k (G_k, p_k)$. Thus \bc $G = ( \{f \in \prod_k G_k \colon \forall k \, f(k)= p_k(f(k+1))\}, \cdot)$,  \ec which is closed and hence  compact group subgroup of the Cartesian product of the $G_k$. We claim that $G \cong (\+F(M), \cdot)$ via the map $\Phi$ that sends $f \in G$ to the filter in $\+ F(M)$ generated by the $^*$cosets $f(k)$, namely \bc $\Phi(f) = \{C\in M \colon \, \exists k \, f(k) \sqsubseteq C\}$.      \ec  It is clear that $\Phi$ is a monomorphism. To show $\Phi$ is onto, given a full filter $x\in \+ F(M)$, for each $k$ there is $f(k)= B_k \in LC(N_k)$ such that $B_k \in x$. Then $f \in G$, and clearly $\Phi(f)= x$.

Note the $\hat N_k$ form a base of nbhds of $1$ in $\+ F(M)$. Since $\Phi^{-1}(\hat N_k)= \{ f \colon f(k) = N_k\}$ and the letter sets form a base of nbhds of $1$ in $G$, we get that $\Phi$ is a homeomorphism. 
%
Thus $\+ F(M)$ is compact. 
\end{proof}

\subsection*{Coarse groupoids versus diagrams,  for profinite groups} \

\n 
We will characterize the coarse groupoids $\+ M(G)$ obtained from    profinite groups $G$ by adding an axiom to the $^*$compactness condition.

Consider profinite  $G=\projlim_k (G_k, p_k)$   where the $G_k$ are finite groups and  each $p_k\colon G_{k+1} \to G_k$ is an epimorphism.  We  say that $\seq {G_k, p_k}$ is a \emph{diagram} for $G$. By the   proof of Prop.\ \ref{prop:profinite}, each diagram for $G$ can be seen as  a coarse groupoid $M$ with $\+ F(M) \cong G$. So a coarse groupoid for $G$ is not unique. Intuitively, they may  be open subgroups of $\+ F(M)$ that $M$ is missing. To avoid this we need another axiom. In the axiom to follow,   ``CC" stands for ``completeness in case of  compactness". We will see that it  implies in the compact case  that each open subgroup of $\+F(M)$ has a ``name".

\medskip
 
\n {\bf Axiom CC.} {\it Let $M$ be a $^*$compact coarse groupoid. 
 Let $N$ be a  normal $^*$subgroup of $M$.

\n If a set $\+ S \sub LC(N)$ is closed under products and inverses, then there is a $^*$subgroup $U$ such that $A \sq U \lra A \in \+ S$,  for each $A \in LC(N)$.}  

\medskip

Clearly $\+ M(G)$ for profinite $G$ satisfies this axiom.  The axiom implies the dual of Prop.\ \ref{fact:standard} in the compact case:
\begin{prop} \label{prop: descr compact} Let $M$ be a $^*$compact coarse groupoid satisfying Axiom~CC. Then $M \cong \+ M( \+ F(M)) $ via the map $A \mapsto \hat A$. \end{prop}
\begin{proof} 
 By Claim~\ref{subset is correct} it suffices to show that the map is onto. 

Firstly,  let $\+ U$ be an open subgroup of $\+ F(M)$. By definition of the topology and Prop.~\ref{prop:normal subgroup},  there is a normal $^*$subgroup  $N$ in  $M$ such that $\hat N \sub \+ U$.  By Prop.\ \ref{prop:profinite},  $\+ F(M)$ is compact, so $\+ U$ is the union of finitely many cosets of $\hat N$.  By Claim~\ref{correct form of left cosets}  each such coset has the form $\hat A$ for some $A \in LC(N)$. Let $\+ S$ be the set of such~$A$ in $LC(N)$. The set $\+ S$ is closed under product and inverses since $\+ U$ is a subgroup, using Claim~\ref{subset is correct}. So there is a $^*$subgroup $U$ as in Axiom~CC. Clearly $\hat U = \+ U$. 

Secondly, given a left coset $\+ B$ of   an open subgroup $\+ U$ in $\+ F(M)$, by Claim~\ref{correct form of left cosets}  we have $\+ B = \hat B$ for some $B \in LC(U)$ as required.
\end{proof}

\subsection*{T.d.l.c.\  groups and $^*$locally compact coarse groupoids} 
\

\n We say that  a coarse groupoid $M$ is \emph{$^*$locally compact} if for each $^*$subgroup  $K\in M$   the coarse  subgroupoid induced on $\{ A \colon \,  A \sq K\}$ is $^*$compact. Note that  $\+ M(G)$ for t.d.l.c.\ $G$ has this property:  by van Dantzig's theorem (that every t.d.l.c.\ group has an open compact subgroup) $\+ M(G)$ is non-empty, and by definition $\+ M(G)$ consists of compact open cosets. 

 We call $M$ \emph{weakly $^*$locally compact} if  for \emph{some}   $^*$subgroup   $K\in M$   the inductive  subgroupoid on $\{ A \colon \,  A \sq K\}$ is $^*$compact.

\begin{prop}  \label{prop:tdlc} Let $M$ be a coarse groupoid.  Then 

  $M$ is weakly $^*$locally compact $\LR$  $\+ F(M)$ is locally compact   \end{prop}

\begin{proof}  

$\LA$:  By van Dantzig's theorem, $\+ F(M)$ has  a compact open subgroup~$\+ L$. Let $K \in M$ be a $^*$subgroup such that $\hat K \sub \+ L$. As above let $M_K$ be the coarse subgroupoid of $M$  on $\{ A \colon \,  A \sq K\}$. Then $\+ F(M_K) \cong \hat K$. So $\+ F(M_K)$ is compact. Hence $M_K$ is $^*$compact by Prop.\ \ref{prop:profinite}. 
 $\RA$. 
Let $M$ be weakly $^*$locally compact via $K$. Then $\+ F(M_K)\cong \hat K$ is compact. Since $\hat K$ is an open subgroup of $\+ F(M)$ this makes $\+ F(M)$ locally compact. 
\end{proof}

\begin{example} (a) If $G$ is countable discrete group, then $\+ M(G)$ consists of the isomorphisms between finite subgroups. 

\n (b) Let $G= ( \mathbb Q_p,+)$. The   proper open subgroups are compact, and are all of the form $U_r= p^{r}\ZZ_p$ for some  $r\in \ZZ$. In this abelian setting each morphism is an endomorphism. The group $G_r$ of endomorphisms  $A \colon \, U_r \to U_r$ has the structure of $C_{p^\infty}$ (the direct limit of the cyclic groups $C_{p^n}$ with the canonical embeddings). Let $f(x) = px$ for $x \in C_{p^\infty}$ and view $f$ as a map $  G_r \to G_{r+1}$. Then for $A \in LC(U_r), B\in LC(U_{r+1})$, the ordering relation $A \sq B$ is equivalent to  $f(A)= B$. 
We see that the coarse groupoid is a bit like a diagram for a profinite group, but goes not only to closer approximations of elements (backwards), but also to less close ones (forward). 

\n (c) $G_d= \Aut(T_d)$ for $d\ge 2$. This is the group of automorphism of the $d$-regular tree $T_d$, defined as an undirected graph without a specified root, first studied by Tits. It is known that each proper open subgroup is compact. Each compact  (open or not)  subgroup is contained in   the stabilizers of a vertex, or  the stabilizers of an edge (which are compact open). See \cite[p.\ 12]{Figa.Nebbia:91}. It would be interesting to describe more of  the structure of $\+ M(G_d)$. 
\end{example}

Recall that in the locally compact setting, the coarse groupoid  $\+ M(G)$ has as a domain the compact open cosets of $G$. We replace Axiom CC from the compact setting by a variant that works in the more general setting.
\medskip 
 
\n {\bf Axiom CLC.}  {\it  Let $M$ be a $^*$locally compact coarse groupoid.  Let $N$ be a   $^*$subgroup of a $M$.  

If a finite set $\+ S \sub LC(N)$ is closed under products and inverses, then there is a $^*$subgroup $U$ such that $A \sq U \lra A \in \+ S$ for each $A \in LC(N)$.  }

\medskip 
Clearly, if $G$ is t.d.l.c.\ then  $\+ M(G)$      satisfies this axiom. We verify  that the axiom characterizes the  $^*$locally compact coarse groupoids obtained in this way.

\begin{prop} Let $L$ be a $^*$locally compact coarse groupoid satisfying Axiom~CLC. Then $L \cong \+ M( \+ F(L)) $ via the map $A \mapsto \hat A$. \end{prop}
\begin{proof} 
 As in Prop.\ \ref{prop: descr compact}, by Claim~\ref{subset is correct} it suffices to show that the map $A \mapsto \hat A$ is onto.  
 
 Firstly let $\+ U$ be a compact open subgroup of $\+ F(L)$. There is $W \in L$ such that $\hat W \sub \+ U$. Let 
 \bc $L_\+ U= \{ A \in L \colon \, \hat A \sq \+ U\}$. \ec
 Clearly $L_\+ U$ is a coarse  subgroupoid of $L$. In $L_\+ U$, $RC(W)$ is finite, so by Prop.\ \ref{prop:normal subgroup} there is $N \sq W$ such that $\+ S:= LC(N)= RC(N)$ in $L_\+ U$.   Clearly the hypothesis of Axiom~CLC applies to $\+ S$, so we get a $^*$subgroup $U$. Then $\hat U = \+ U$.

 Secondly, given a left coset $\+ B$ of   an open subgroup $\+ U$ in $\+ F(L)$, by Claim~\ref{correct form of left cosets}  we have $\+ B = \hat B$ for some $B \in LC(U)$ as above.
\end{proof}

\subsection*{$\+ M$ and $\+ F$ as functors}

We will view closed subgroups of $\S$   (also called non-Archimedean groups)    as  a category where the morphisms $f\colon G \to H$ are the continuous   epimorphisms. The kernel  of $f$ is a closed normal subgroup. If   $G$ is compact/locally compact then  $H$ has the same property.

\medskip

\n {\bf Covariant functor, for the   locally compact case.}  
The open mapping theorem for Hausdorff topological groups  states that a surjective continuous homomorphism of a  $\sss$-compact   group  onto a  Baire (e.g., a locally compact)  group  is an open mapping. (This is proved using Baire category.) Each (separable) t.d.l.c.\ group is $\sss$-compact again by van Dantzig.
 So if $G$ is locally compact and $f\colon G \to H$ is  onto,  then  for each compact open $A \sub G$, the image $f(A)$ is open (and of  course compact) in $H$. 

On t.d.l.c.\ groups with continuous epimorphisms we can now  view the operator $\+ M$ as a covariant functor $\+ M_+$. If $f\colon G \to H$ then $\+ M_+(f) \colon \+ M(G) \to \+  M(H)$ is given by $A \mapsto f(A)$. 
Here we view coarse groupoids as a weak category  $\+ C_{weak}$ where the morphisms $r \colon \, M \to N$ preserve the groupoid structure and the partial order in the forward direction only.  
\medskip

\n {\bf Contravariant functor.} 
If $f\colon G \to H$ is an epimorphism of non-Archimedean groups,   then we have a map $\+ M_-(f) \colon \, \+ M(H) \to \+ M(G)$, where $\+ M_- (f)(A)= f^{-1}(A)$.  This clearly preserves the inductive groupoid structure. 
 We want to identify the right type of morphisms  on coarse groupoids so that $\+M$ is a contravariant functor with inverse $\+ F$ (suitably extended to morphisms) for the classes of groups of interest.   Consider a  map $\+ M_-(f)$ above  and let  $R \sub \+M(G)$ be its range.  $R$ consists of all open [compact] cosets of subgroup of $G$ that contain  the kernel of $f$. So, 

\bi \item $R$  is  closed upwards, and 

\item $R$ is closed under conjugation of subgroups. Thus if $\ex A  [ _U A_V]$ and $U \in R$ then $V \in R$. \ei

So we have to take the strong category $\+ C_{str}$ with objects the (countable)   coarse groupoids and with morphisms   $q \colon \, N \to M$ that preserve the groupoid and the meet semilattice structure (in particular, they preserve the ordering and hence have to be $1-1$) and have   range with  the two  properties above. For such an $q \colon \, N \to M$   we can define  $\+ F(q) \colon \, \+ F(M) \to \+ F(N)$   by $x \mapsto q^{-1}(x)$.
 
\begin{claim} Suppose $q\colon N \to M$ is a morphism in $\+ C_{str}$. 

\n Then $\+ F(q) \colon \+ F(M) \to \+ F(N)$   is a continuous epimorphism.  \end{claim} 


\section{Nies: Closed subgroups of $\S$ generated by their permutations of finite support}
\newcommand{\FS}{SF(\omega)}
Let $\FS$ denote the groups of permutations of $\omega$ that have finite support. We note that $\FS$ plays a role in $\S$ similar to the role that $\mathbb Q$ plays in $\R$.
 Clearly $\FS$ is dense in $\S$. The inherited topology has as a basis of neighbourhoods  of $1$ the open subgroups $\FS \cap U_n $ of $\FS$, where $U_n \le_o \S$ is the pointwise stabilizer of $\{0, \ldots, n\}$. For example the subgroup $T$ algebraically generated by  $\{(01) (2n \, 2n+1) \colon n \ge 1\}$ of $\FS$  is not closed;  its closure  is   $T \cup \{(01)\}$. 

Given a closed subgroup $G$ of $\S$, it is of interest to find out how much of~$G$ can be recovered from  the closed  subgroup $G \cap \FS$ of $\FS$.  B.\ Majcher-Iwanov~\cite{Majcher:20}  called $G \cap \FS$ the \emph{finitary shadow} of $G$.
Often $G \cap \FS$  will be trivial, for instance if $G$ is torsion free. 

Note that a closed subgroup $G$ of $\S$  is compact iff each $G$-orbit is finite. 
Majcher--Iwanov~\cite{Majcher:20}  studied the distribution of finitary shadows of compact subgroups of $\S$ within the subgroup lattice  of $\FS$, and made connections to cardinal characteristics. 

One  can also start from a   subgroup $H$ of $\FS$ and study how it is related to    its closure $G= \ol H$ in $\S$.   (Closures will be taken    in $\S$ unless otherwise stated.) 

Firstly, for each open subgroup $U$ of $H$, its closure $\ol U$ is open in $G$. For, suppose $U_n \cap H \le U$. Since  $U_n$ is closed  we have  $\overline {U_n \cap H} = U_n \cap \ol H$.  Hence  $U_n \cap G \le \ol U$, so $\ol U$ is open in $G$.

Conversely, each   open subgroup of $G$ is   given as the closure of the   open  subgroup $L \cap H$ of $H$:

\begin{lemma} Let $H \le \FS$ and let $G= \ol H$ be the closure of $H$ in $\S$. Then  $L = \ol {L \cap H}$ for each open subgroup $L$ of $G$.  \end{lemma}

\begin{proof} Since $L$ is closed in $G$, we only need to verify the inclusion $\sub$.

Let $g \in L$.
Since $L$ is open in $G$, there is $k$ such that for each $v$, if $g v^{-1}\in U_k$ then $v \in L$. Now let $n \ge k$ be arbitrary.  Since $G= \ol H$, there is $r \in H$ such that $g r^{-1} \in U_n$. So $r \in L$. This shows $g \in \ol{L \cap  H}$.
\end{proof} 

\begin{remark} \label{rem:sum} To summarize, there is a  natural isomorphisms between the open subgroups $L$ of $G$ and the open subgroups $U$ of $H$,  given by $L \to L \cap H$, with inverse $U \to \ol U$. \end{remark}

\subsection*{The compact case}
   The  textbook on permutation groups by Dixon and Mortimer~\cite[Lemma 8.3D]{Dixon.Mortimer:96} contains a proof of   the following equivalences for a group $H \le \FS$:

\medskip
(i) every $H$-orbit is finite (i.e., $\ol H$ is compact)

(ii) $H$ only has finite conjugacy classes (such $H$ is called an FC-group)

(iii) $H$ is residually finite.
\medskip

The implication (i)$\to$(ii) is pretty elementary: for each $x \in H$ there is a finite $H$-invariant set $\Delta$ such that the support of $x$ is contained in $\Delta$. The number of conjugates of $x$ is then bounded by the number of conjugates of $x\mid \Delta$ (restriction to $\Delta$) in $S_\Delta$.  (ii)$\to$(iii) also easy, the remaining implication (iii)$\to$(i) is harder. It was proved in \cite[Thm.\ 2]{Neumann:76}. 

It would be interesting to determine the complexity of the topological   isomorphism relation  for closed subgroups $H$ of $\FS$ that are FC-groups. Since $H = \ol H \cap \FS$, where the closure is taken in $\S$ as usual, it is Borel reducible to topological isomorphism of profinite groups, which by Kechris et al.\ \cite{Kechris.Nies.etal:18} is Borel isomorphic to countable graph isomorphism. However,  the groups employed  there  to reduce graph isomorphism do not have a finitary  shadow satisfying the conditions  above.   \cite{Kechris.Nies.etal:18} uses an extension to the topological setting  of Mekler's construction to code  a countable graph into a countable step 2 nilpotent group of exponent a fixed odd prime. Mekler's  groups being FC would mean that the coded graph is co-locally finite (each vertex connected to cofinitely many vertices). But for technical reasons related to the definability of the vertex set, Mekler's method  only can encode    graphs that are triangle and square free, and such graphs are of course not co-finitary.

\subsection*{The oligomorphic  case} 
 
Recall that a group $H \le \S$ is called oligomorphic if for each $r$ its action on $\omega$ has only finitely many $r$-orbits. It is open whether $E_0$ can be Borel reduced to topological isomorphism of closed oligomorphic subgroups of $\S$. See Section~\ref{s:OQ nonarch}, or~\cite{Nies.Schlicht.etal:19} for background.

For $H \le \FS$, note that $H$ is oligomorphic  iff $\ol H$ is.  The hope was that  one can reduce $E_0$ to topological isomorphism of closed subgroups of $\FS$, which should be easier to control. To be useful, this actually assumed an affirmative answer to the following question: \begin{question} \label{q:olig FS} Is the following true?

Let $H_0, H_1 \le_c \FS$ be oligomorphic. Then \bc $H_0, H_1$ are homeomorphic $\LR$  $\ol H_0, \ol H_1$ are homeomorphic.  \ec
\end{question}
For the implication $\RA$, we note that by Remark \ref{rem:sum} the structures of open cosets for $\ol H_i$ and $H_i$ are isomorphic for $i= 0,1$. That is $\+ M(\ol H_i) \cong \+ M(H_i)$.  By hypothesis $\+ M(H_0) \cong \+ M(H_1)$. Since the $\ol H_i$ are closed oligomorphic subgroups of $\S$,  $\+ M(\ol H_0) \cong \+ M(\ol H_1)$ implies that they are topologically isomorphic by~\cite{Nies.Schlicht.etal:19}.

  However, the intended application  doesn't not    work, because  oligomorphic subgroups  $H$ of $\FS$ are very  restricted, and in particular there are only countably many up to isomorphism. 

Some examples of such groups  $H$   come to mind. First take  permutations of finite support preserving the  evens. This group is topologically isomorphic to $\FS\times \FS$.  More generally, take  a finite power of $\FS$. Another type of example is  letting $H$ be the automorphism group  of a countably infinite structure obtained taking a finite structure $M$, and have an equivalence relation $E$ with a ``copy" of $M$ on each class. For instance,  take one unary function symbol $f$, and let $M$   be a finite  cycle given by $f$. If $\phi$ is the permutation that is the union of all these cycles of fixed length, then $H$  is the centralizer of $\phi$ in $\FS$.

Let $R$ be the random graph and $Q$ the random linear order.  In contrast, $\Aut(R)$ and $\Aut(Q)$ have no nontrivial members of finite support at all (not hard to check). Even $G= \Aut (E)$ where $E$ is an equivalence relation with just  two infinite classes, is not the closure of   $G\cap \FS$ because an automorphism of finite support cannot leave an equivalence class. More generally, if $R$ is a definable relation and $\phi(a)= b$ where $\phi$ is  an automorphism of finite support, then $R(a) $ almost equals $R(b)$.

Towards a full classification of oligomorphic subgroups $H$ of $\FS$, we use some structure theory of subgroups of $\FS$ developed in the 1970s by Peter Neumann, e.g.\ \cite{Neumann:76}, and in two short papers by Dan Segal, independently.   Here I'm mostly using the notes on finitary permutation groups by Chris Pinnock, available at \url{chrispinnock.com/phdpublications/}. Unless otherwise stated references to theorems etc refer to those notes.

Let us begin by  assuming that $H \le \FS $ is 1-transitive.  
 The Jordan-Wielandt Theorem says for such a group $H$ that if 
  $H$ is primitive then $H= \FS$ or $H$ equals the group of alternating permutations of finite support (which has index 2 in $\FS$). So we can assume otherwise. Let $B$ be a  block of imprimitivity.   $B$ has to be finite (Lemma 2.2 in Pinnock), simply because there is  a permutation $\tau \in H$ moving $B$ to a disjoint set $\tau(B)$, so that $B$ is contained in the support of $\tau$. The equivalence relation $E$ with classes made up of the translates of $B$ is $H$-invariant, and hence a union of 2-orbits. Since there are only finitely many 2-orbits, there must be a maximal block of imprimitivity. Then $H$ acts primitively on $\Omega/E$ and each induced permutation there has finite support. So it acts highly transitively on it. This is the  second type of example above, the same  finite model in each equivalence class of $E$. Note that $H$ is a subgroup of $G_0 \wr R$ where $R= \FS $ or $R$ is the alternating group (Thm 2.3).
  
  If $H$ is not 1-transitive we are in the case of the first example above. It is known that  we get only finite products of 1-transitive groups, and  a finite group.

 For a more detailed treatment, see Iwanow~\cite{Ivanov:02} who works in the more general setting of automorphism groups of countable saturated structures. He  uses the concept of cell, a permutation group that preserves an equivalence relation with all classes finite, and induces the full symmetric group on the equivalence classes. A 2-cell is    permutation group $G\le_c \S$ such that there exists a partition $\omega $   into $G$- invariant classes $Y_i$ such that for any infinite $Y_i$  the group   induced by $G$ on$Y_i$  is a cell and is also induced by the pointwise stabilizer of the complement of $Y_i$.  Each oligomorphic 2-cell is a finite product of cells and 
 a finite algebraic closure of the empty set.
Their number is countable as already mentioned. 
Each cell is a finite cover of a pure set in the sense of Evans, Macpherson, Ivanov, Finite covers, 1997. These covers have been classified.

\section{Harrison-Trainor and Nies: $\Pi_r$-pseudofinite groups}
The authors of this post worked at the (now defunct)  Research Centre Coromandel  in July. They started from some notes of Dan Segal (2014), which in turn are based on \cite{Houcine.Point:13}.  The main concept in these notes is this. \begin{definition} A group $G$ is called \emph{pseudofinite} if each first-order sentence true in $G$ also holds in some  finite group.\end{definition} In the first two sections we present  a simplified account of   some material in  the Segal  notes. Our account  is somewhat more general than the notes, firstly as  it  takes into account the quantifier complexity of the sentences for   which  a group needs to be pseudofinite, and secondly because  a lot of this work can be carried out  in the   setting of an arbitrary finitely axiomatised  variety of algebraic structures.

 Throughout, let $G$ be an   algebraic structure  in a language $L$ with    finitely many function and constant symbols, and no relation symbols besides equality.  The $\Pi_0$ and the $ \Sigma_0$ sentences are the quantifier free ones, thought of as disjunctions of conjunctions of equalities and inequalities.  Note that satisfaction of  $\Pi_1$ sentences is closed under taking substructures.
\begin{definition} \bi \item[(i)]  For $r\in \NN$ we say that an $L$-structure $G$ is  \emph{$\Sigma_r$-pseudofinite} if  $G$ is a model of the $\Sigma_r$-theory of finite $L$-structures. Equivalently,  each $\Pi_r$ sentence that holds in $G$ also holds in some finite $L$-structure.  \item[(ii)] Similarly, we define \emph{$\Pi_r$-pseudofiniteness} of an  $L$-structure $G$. \ei \end{definition}

\begin{definition} \label{conv:basic} We  say that an  $L$-structure $G$ has \emph{named generators} if $G$  is $d$-generated for some $d \ge 1$, and  the language $L$ contains finitely many constant symbols $\ol c= (c_1, \ldots, c_d)$ naming such generators of $G$. \end{definition}
In this case  witnesses for outermost existential quantifiers be named by terms. So we have

\begin{fact} \label{fact:psfinite vary} Let $G$ be an $L$-structure with named generators. Let $\wt G$ be the structure with the constants naming the generators  omitted. Let $ r \ge 1$. The following are equivalent.

(i) $\wt G$ is $\Pi_{r+1}$-pseudofinite 

(ii)  $  G$ is $\Pi_{r+1}$-pseudofinite

(iii) $G$ is $\Sigma_r$-pseudofinite.
\end{fact}

\begin{proof} (ii)$\to$(i) and (i)$\to$(iii) are straightforward, and don't rely on the fact that the $c_i$ name generators. 

For (iii)$\to$(ii), suppose a sentence $\theta\equiv \ex \ol x \, \xi(\ol x)$ holds in $G$, where $\ol x $ denotes $ (x_1, \ldots, x_m)$ and $\xi(\ol x)$ is $\Pi_r$. Since the $c_i$ name generators, we can pick $L$-terms $t_j$ without free variables as witnesses. The sentence $\xi(t_1, \ldots, t_m)$ is $\Pi_r$ and holds in $G$. So it holds in some finite $L$-structure. Hence $\theta$ holds in that structure. \end{proof}  

The $\Sigma_1$-theory of finite groups equals the $\Sigma_1$-theory of the trivial group, which is decidable. Every  group satisfies the $\Sigma_1$-theory of the trivial group and hence is $\Sigma_1$-pseudofinite. So the notion $\Sigma_1$-pseudofiniteness really only makes sense for groups with   named generators. 

In contrast, the $\Pi_1$-theory of finite groups is hereditarily undecidable by a result of Slobodskoi~\cite{Slobodskoi:81}. However, it has infinitely many  decidable completions, e.g. the theory of $(\ZZ^n, +)$, for each $n \ge 1$. To see this, let $H$ be the ultraproduct of all cyclic groups of prime order. $H$ is abelian, torsion free, not f.g.\  and satisfies the theory of finite groups. Any subgroup of $H$  satisfies the $\Pi_1$-theory of $H$. To see that  $H$ has infinite rank,  let $R_i$, $i \in \NN$ be disjoint infinite sets of primes, let $f_i(p)= 1$ for $p \in R_i$ and $0$ otherwise. Then the $[f_i]$ are linearly independent over $\ZZ$.

\subsection{$\Sigma_1$-embedding of $G$ into an ultraproduct of witness structures} \ 

\n 
The following construction  is based on Houcine and Point~\cite{Houcine.Point:13}. We fix some variety $\+ V$ of $L$-structures, such as groups. Suppose $G$ is as in Definition~\ref{conv:basic}, and $G$ is $\Sigma_1$-pseudofinite. Let $\seq{ \phi_i'}\sN i$ be a  list of  the $\Pi_1$-sentences that hold in $G$, with $\phi_0'$ being the conjunction of the finitely many axioms for $\+ V$. Let $\phi_n=\bigwedge_{i\le n} \phi'_i$. By hypothesis on $G$ there is a finite $L$-structure $E_n$ (called a witness structure) such that $E_n \models \phi_n$. Since we are only considering $\Pi_1$-sentences, we may assume that each $E_n$ is generated by $\ol c^{E_n}$. 

Let  $\+ U$ be a free ultrafilter on $\NN$. Let $E = \prod_n E_n /\+ U$ be the corresponding ultraproduct.   
Define an embedding  $\beta \colon G \to E$ by setting for each $L$-term $s$

\bc $\beta (s(\ol c^G)) = s(\ol c^E)$. \ec
Note that $\beta$ is well-defined and $1-1$: for each terms $s,t$, if $G \models s(\ol c)= t(\ol c)$ then the sentence $s(\ol c)= t(\ol c)$ equals    $\phi'_i$  for some $i$, and so  $E_n \models s(\ol c)= t(\ol c)$ for each $n \ge i$. One argues similarly for the case  $G \models s(\ol c)\neq t(\ol c)$.  So $\beta$ is a monomorphism of $L$-structures.

\begin{fact} \label{fact:elembed} Suppose an  $L$-structure $G$ with named generators as in Definition~\ref{conv:basic}  is  $\Sigma_1$-pseudofinite. The map $\beta \colon G \to E$ preserves satisfaction of $\Sigma_1$-formulas in both directions. Thus, 
for each $p_1, \ldots, p_k \in G$ and each quantifier free $L$-formula $\theta(x_1, \ldots, x_k, \tilde y)$,  
\bc $G \models \ex \tilde y \ \theta(p_1, \ldots, p_k, \tilde y) \LR E \models \ex \tilde y \  \theta(\beta(p_1), \ldots, \beta(p_k), \tilde y)$. \ec
\end{fact}
\begin{proof}  We may incorporate the parameters into $\theta$, so we may assume that there are none. Let $\phi \equiv \ex \tilde y \ \theta(  \tilde y) $.  For the nontrivial implication suppose that $E \models \phi$. If $G \models \lnot \phi$, then $E_n \models \lnot \phi$ for almost all $n$, contradiction. So $G \models \phi$.
\end{proof}

\subsection{The number of factors needed for being in a verbal subgroup}
In this section we only consider the case that $G$ is a group. We will explain the result of Segal that if $G$ is $\Pi_2$-pseudofinite  then the verbal subgroups $G^{(n)}$ and $\la G^q\ra$, $n, q \ge 1$,  are  definable in $G$ by a positive  $\Sigma_1$ formula.

Below  $G$ satisfies Definition \ref{conv:basic} where $d$ denotes the number of   generators. The language $L$ consists of symbols for the group operations, the neutral element, and constants $c_1, \ldots, c_d$ naming the generators.

\begin{definition} Let $d,r \ge 1$. Let  $w$ be a term with free variables $x_1, \ldots, x_k$ in the language of group theory, identified with an element of the free group $F_k$. We say that $w$   is \emph{$r$-bounded for $d$} if for each finite $d$-generated group $H$, we have 

\bc $w(H)= w^{*r}(H)$. \ec
Here for any group $H$, by $w(H)$ one denotes the usual verbal subgroup which is generated by the values of $w$, and $w^{*r}(H)$ is the set of products of up to~$r $ many  values of $w$ or their inverses. \end{definition}

Nikolov and Segal \cite[Thm 1.7]{Nikolov.Segal:07}, or \cite[Thm.\ 1.2]{Nikolov.Segal:12},  show that this holds for the terms  $x^q$, where $q \ge 1$, and iterated commutators such as $[x,y]$ or $[[x,y],[z,w]]$. For such terms, the result below (for pseudofinite groups)  was established in \cite[Prop.\ 3.8]{Houcine.Point:13}, but an additional hypothesis was needed for the case of terms  $x^q$.

\begin{prop} \label{prop: small products} Suppose a  $d$-generated group $G$ with named generators as in  Definition \ref{conv:basic}  is $\Sigma_1$-pseudofinite w.r.t.\ the language $L$. Suppose further that $w$ is $r$-bounded for $d$. Then $w(G)= w^{*r}(G)$. 

By Fact~\ref{fact:psfinite vary} the hypothesis holds  if $G$ is $\Pi_2$-pseudofinite w.r.t.\  the language of group theory. In this case $w(G)$ is $\Sigma_1$-definable without parameters in that language. \end{prop}

\begin{proof} Write $g_i= c_i^G$ and $\ol g=(g_1, \ldots, g_k)$. Suppose $t(\ol g)\in w(G)$ for some term $t$. That is, for some $m$, there are $k$-tuples $\ol z_j $ of elements of $G$ and $\epsilon_j = -1, 1$,  for $1 \le j \le m$,  such that \bc $t(\ol g) = \prod_{1 \le j\le m} w(\ol z_j)^{\epsilon_j} $. \ec
Then $E $ satisfies the $\Sigma_1$-sentence $ \ex \ol z_1 \ldots \ex \ol z_m \, [ t(\ol c) = \prod_{1 \le j\le m} w(\ol z_j)^{\epsilon_j}] $, and hence the ultrafilter $\+ U  $ contains the set $S$ of those $n$ such that $E_n$ satisfies this sentence. By hypothesis on $w$  and since $E_n$ is $d$-generated by the $c_i^{E_n}$, for each $n\in S$ there is a choice of $\eta_\ell = -1, 1$, for $1 \le \ell \le r$,  such that 
\bc $E_n \models  \ex \ol y_1 \ldots \ex \ol y_m \, [ t(\ol c) = \prod_{1 \le j\le r} w(\ol y_j)^{\eta_j} ]$. \ec
Since $\+ U$ is an ultrafilter, there must be one choice $(\eta_\ell)_{1 \le \ell \le r}$ so that the $n$ where this choice applies form a set in $\+ U$. Hence  for this choice we have
\bc $E \models  \ex \ol y_1 \ldots \ex \ol y_m \, [ t(\ol c) = \prod_{1 \le j\le r} w(\ol y_j)^{\eta_j} ]$. \ec
By Fact~\ref{fact:elembed} this implies that $t(\ol g) \in w^{*r}(G)$ as required.
\end{proof}
 As Houcine and Point point out \cite[Lemma 2.11]{Houcine.Point:13}, parameter definable quotients and subgroups of pseudofinite groups are again pseudofinite. Also, finitely generated pseudofinite groups that are of fixed exponent, or soluble, are finite \cite[Prop.\ 3.9]{Houcine.Point:13}. 
So, $G$ is an extension of a pseudofinite group,  the verbal subgroup given by a term as above, by a finite group satisfying the law given by the term.
\subsection{Effectiveness considerations}

We show that a variant of the basic construction above leading to Fact~\ref{fact:elembed} is effective using   the structure $G$ as an oracle.

\n \emph{Effective ultraproducts.} Given a  sequence $\seq{E_n}$ of uniformly computable structures for the same finite signature, and a non-principal ultrafilter $\+ U$ for the Boolean algebra of recursive sets, one can form the structure 
$E = \prod_{rec} E_n/{\+ U}$. It consists of $\sim_\+ U$-equivalence classes of recursive functions $f$ such that $f(n) \in E_n$ for each $n$. Here $f \sim_\+ U g$ denotes that functions  $f,g$ agree on a set in $\+ U$. We interpret the  relation and functions symbols on $E$ as usual. 

A   version of Los' Theorem restricted to existential formulas holds. 
\begin{fact}
For each   formula $\theta(x_1, \ldots, x_m ) \equiv \ex \wt y \, \xi(x_1, \ldots, x_m , \wt y)$ with $\xi$ quantifier free, and each $[f_1], \ldots, [f_m]\in E$, 

\bc $E \models \theta([f_1], \ldots [f_m]) \LR R=\{ n \colon E_n \models \theta(f_1(n), \ldots f_m(n))\} \in \+ U$. \ec \end{fact}
\begin{proof} 
Note that  the set $R$ is computable. For the implication from left to right, suppose $[h_1], \ldots, [h_k]$  are witnesses for  $E \models \theta([f_1], \ldots [f_m])$,  for  recursive functions $h_1, \ldots, h_k$.This shows that  $R \in \+ U$ by definition  of the ultraproduct.

For the implication from right to left, note that when $n$ is in the computable set $R$ one  can search for the witnesses $w_n$ in $E_n$. So one can define computable witness functions $h_1, \ldots, h_k$ for $E \models \theta([f_1], \ldots [f_m])$,  assigning a vacuous value in $E_n$ in case $n \not \in R$. \end{proof}

 To obtain an effective version of  Fact~\ref{fact:elembed},   let $ \psi_i$ be an effective list of all the $\Pi_1$ sentences in $L$, where $\psi_0$ is the conjunction of the laws of $\+ V$. We modify the construction  of the $E_n$ as follows.
 Given $n$, look for the least  stage $s$   and a finite $L$-structure, called  $E_n$, such that  $E_n$ satisfies each $\psi_i$, $i \le n$, that has not been     shown to fail in $G$ by stage $s$, in the sense that a counterexample has been found among the first $s$ elements of $G$. 
 
 Now define the restricted ultraproduct $E$ as above. The argument for the fact can be carried out  as before:  Suppose that $E\models \phi$ for an existential sentence $\phi$. If $G \models \lnot \phi$, then $E_n \models \lnot \phi$ for almost all $n$. This contradicts the weak version of Los theorem above.
 
 The argument of Prop~\ref{prop: small products} also works with this restricted ultraproduct. Note that the set $S$ in the  proof of Prop~\ref{prop: small products}  is computable since the $E_n$ are finite and given by strong indices. The set of $n$ for which a particular choice of $\eta_\ell$ works is also computable. 
 
 The upshot: if $G$ is computable, then we have a canonical ultraproduct version $E$ of $G$, with a $\Sigma_1$ elementary embedding, and this version $E$ is in a sense effective as well, except for the ultrafilter, which is necessarily high in a sense specified and proved in a 2020 preprint by Lempp, Miller, Nies and Soskova available at \url{www.cs.auckland.ac.nz/research/groups/CDMTCS/researchreports/download.php?selected-id=769}.

%
 \part{Computability theory and randomness}
   \newcommand{\cost}{\mathbf{c}}
\section{Greenberg, Nies and Turetsky: Characterising SJT reducibility}
\subsection{The basic concepts}
\subsubsection{SJT- reducibility and its equivalents}
 SJT-reducibility was introduced in 
\cite[Exercise 8.4.37]{Nies:book}.
\begin{definition}[Main: SJT-reducibility]  For an oracle $B$, a  $B$--c.e.\ trace is a u.c.e.\ in $B$ sequence $\seq{T_n}\sN n$ of finite sets. For a function $h$, such a trace   is $h$-bounded if $|T_n| \le h(n)$ for each $n$. A set $A$ is jump-traceable if there is a computably bounded $\ES$--c.e. trace $\seq{T_n}\sN n$ such that $J^A(n)$ is in $T_n$ if it is defined.

  For sets $A,B$, we write $A \le_{SJR} B $ if for each order function $h$, there is a $B$--c.e., $h$-bounded trace for $J^A$. 
\end{definition}
This is transitive by argument similar to \cite[Theorem 3.3]{Ng:10}.

(Question: Is ``strongly superlow" reducibility equivalent to $\le_{SJR}$? This was also mentioned in  \cite[Exercise 8.4.37]{Nies:book}, there shown to be at least as strong.)

\begin{proposition} For each $K$-trivial  set $B$,  there is a c.e.\ $K$-trivial set  $A$ such that $A \not \le_{SJR} B$. \end{proposition}
\begin{proof} For some fixed computable function $h$ there is a functional  $\Psi$ and $K$-trivial set $A$ such that $\Psi^A$ has no c.e.\ trace bounded by $h$, by a result of \cite{Cholak.Downey.ea:08}. (Also see \cite[8.5.1]{Nies:book} where $h(n)= 0.5 \log \log n$ works.) Encoding $\Psi$ into $J$, we get a fixed  computable function $g$ so that the statement holds for $J$ and $g$ instead of $\Psi$ and $h$.Relativizing to $B$ we can retain the same $g$, so    for each $B$ there is a $K$-trivial in $B$ set $A$ such that $J^{A \oplus B}$ does not have a $B$-c.e.\ trace bounded by $g$. In particular, $A \not \le_{SJR} B$.

If $B$ is $K$-trivial, it is low for $K$, and hence $A$ is also $K$-trivial. Finally, there is $K$-trivial c.e.\ set $\hat A \ge_T A$, so we can make $A$ c.e.
\end{proof}
\begin{definition}
Let $c$ be  a cost function. For sets $A,B$, we write $A \models_B c$ if there is a $B$-computable enumeration of $\seq{A_s}$ satisfying $c$.
\end{definition}

\begin{definition}[Benign cost functions] A cost function~$\cost$ is \emph{benign} \cite{Greenberg.Nies:11} if from a rational  $\epsilon>0$, we can compute a bound on the length of any sequence $n_1 < s_1 \le n_2 < s_2 \le \cdots \le n_{\ell} < s_{\ell}$ such that $\cost(n_{i},s_{i})\ge \epsilon$ for all $i\le \ell$. For example, $\cost_{\Omega}$ is benign, with the bound being $1/\epsilon$. \end{definition}

Conjecture that strengthens proposition above: For each benign $c$, for each    $B\models c$,  there is a c.e.\    $A \models c$ such that $A \not \le_{SJR} B$.

Conjecture for each noncomputable c.e.\ $E$ there are c.e. $\le_{SJR}$-incomparable $A,B \le_T E$. 

Fact $\le_{SJR}$ is $\SI 3$ on the $K$-trivials. Density might be easy on the $K$-trivials using this. 

\subsubsection{Two  relevant randomness notions}
\begin{definition}[Demuth randomness] \label{ref:  Demuth}  A  \emph{Demuth test}  is a seqence  $\seq {G_m}\sN m$ of open subsets of $\cantor$  such that  $\leb G_m \le \tp{-m}$ and there is  an $\omega$-c.a.\ function $p$ such that $G_m = \Opcl{W_{p(m)}}$. Since the function $p$ is $\omega$-c.a.,  there are a computable approximation  function $p(m,s)$ and a computable bound $b$  such that $\lim_s p(m,s) = p(m)$ and the number of changes is bounded by $b(m)$. One writes $G_m[t]$ for $ \Opcl{W_{p(m,t),t}}$, the approximation of the $m$-th component at stage $t$. One may assume that $\leb G_m[t] \le \tp{-m}$ for each $t$ by cutting  off the enumeration of $G_m$ when it attempts to exceed that measure. One says that $Z$  is \emph{Demuth random} if for each such test, one has  $Z \not \in G_m$ for almost every $m$. 

\end{definition}

\begin{definition}[Weak Demuth randomness]  \label{ref: nested Demuth}
A \emph{nested Demuth test} is a Demuth test  $\seq {G_m}\sN m$  such that  $G_m \supseteq G_{m+1}$ for each   $m$. One says that $Z$ is \emph{weakly Demuth random} if $Z \not \in \bigcap_m G_m$ for each nested Demuth test $\seq {G_m}\sN m$.
Replacing $G_m[t]$ by $\bigcap_{i\le m} G_i[t]$ (and noting that the number of changes remains computably bounded),  one may assume that the approximations are nested at each stage $t$, i.e., $G_m [t]\supseteq G_{m+1}[t]$ for each $m$. 
 \end{definition}

\subsection{Equvalent characterizations of $\le_{SJR}$}
 \begin{theorem} The following are equivalent for $K$-trivial  c.e.\ sets $A,B$. 

\begin{itemize}
\item[(a)] $A \le_{SJR} B$
\item[(b)] $A \models_B \cost$ for every benign cost function $\cost$ 
\item[(c)]   $A \le_T B \oplus Y$ for each ML-random set $Y$   that is  not  weakly Demuth random
\item[(d)]  $A \le_T B \oplus Y$ for each ML-random set $Y \in \+ C$, where $\mathcal C$ is   the class of the  $\omega$-c.a.,    superlow, or   superhigh sets. \end{itemize}
\end{theorem}

We remark  that the implication (b) $\RA$ (a) was essentially obtained by Greenberg and Nies \cite[Prop.\ 2.1]{Greenberg.Nies:11}. They  proved the following. Let $A$ be a c.e., jump-traceable set, and let $h$ be an order function. Then there is a benign cost function $c$ such that if $A$ obeys $c$, then $J^A$ has a c.e.\ trace which is bounded by $h$. Suppose now that  $A \models_B c$. Apply the argument in the proof of \cite[Prop.\ 2.1]{Greenberg.Nies:11}  to a $B$-computable enumeration of $A$  showing this. Then the $h$-bounded trace  $\seq{T_n}$ for $J^A$ constructed there is c.e.\ relative to $B$.

For $\+C \sub \MLR$ define $\+ C^\diamond$. By the implication (a) $\RA$ (c) we obtain:  
\begin{corollary} Let $\+C $ be a nonempty class of ML-randoms that contains no weakly Demuth random. Then $\+ C^\diamond$ is downward closed under $\le_{SJR}$. \end{corollary}
For instance, let $\+ C  = \{\Omega_R\}$ for a co-infinite computable set $R$. This shows that the subideals of $K$-trivials considered in \cite{Greenberg.Miller.etal:19,Greenberg.Miller.etal:22} are SJT-ideals.

We will prove the  implications of the theorem   in a cycle, starting from (b). The implications (b) $\RA$ (c) and (c) $\RA$ (d) rely   on literature results, or standard methods from the literature. The implication (d) $\RA$ (a) combines methods from \cite{Nies:11} and \cite{Bienvenu.Downey.ea:14}. The implication (a) $\RA$ (b) uses  the box promotion method; for background see \cite{Greenberg.Turetsky:14,Greenberg.Turetsky:18}.

\subsubsection{Proof of   (b) $\RA$ (c)}  Hirschfeldt and Miller showed that for each null $\PI 2$ class $\mathcal H\sub \cantor$  there is a cost function $\cost$ such that $A \models \cost$ and $Y \in \MLR \cap \mathcal H$ implies $A \le_T Y$; see \cite[proof of 5.3.15]{Nies:book} for a proof of  this otherwise unpublished result.

Suppose that a ML-random set $Y $ fails a nested Demuth test $\seq{G_m}$. Then $Y \in \mathcal H = \bigcap_m G_m$ which is a $\PI 2$ class. We will apply the method of Hirschfeldt and Miller   but  incorporating the additional oracle set $B$, and using the particular   representation of  the $\PI 2$ null class by a nested Demuth test  to ensure that the cost function $\cost$ is benign.  Also see post on weak Demuth randomness by \Kuc\  and Nies in \cite{LogicBlog:11}

Let $p(m)$ and its approximation $p(m,t)$ be as in Definition~\ref{ref: nested Demuth} of nested Demuth tests. We define the benign cost function $\cost$ as follows.  Define for $k \le t$

\begin{eqnarray*} r(k,t)&=& \min\{m \colon \, \exists s. k< s\le t \, [p(m,s-1) \neq p(m, s)]\} \\
			V_{k,t}&= & \bigcup_{k \le s \le t} G_{r(k,s)}[s] \\
			\cost(k,t)&=& \leb V_{k,t}. \end{eqnarray*} 
			
Clearly $r(k-1,t) \ge r(k,t)$ for $k>0$, hence the $V_{k,t}$ are nested and $\cost(k,t)$ is nonincreasing in $k$. Similarly, $\cost(k,t)$ is nondecreasing in $t$. Note that if $p(m,s)$ has stopped changing by a stage $ k$,  then $r(k,t) > m$ for each $t \ge k$, and hence $V_{k,t}$ is contained in the final version $G_m$ of the $m$-th component. 
By the conventions in Definition~\ref{ref: nested Demuth} above,  if $\cost(k,t) > \tp{-m}$ then $p(m,s-1)\neq p(m,s)$   for some $s$ in the interval $(k,t]$. 
Since the number of changes of $p(m, s)$  is bounded computably in $m$, this shows that the cost function $\cost$ is indeed benign.

Suppose now that $A \models_B \cost$ via a $B$-computable approximation $\seq{A_s}$.   To show $A \leT Y\oplus B$  for each ML-random set $Y \in
\mathcal \bigcap_m G_m$,  
we enumerate a
Solovay test $\+ S$ relative to $B$, i.e., a uniformly $\SI 1 $ sequence $\seq {\+ S_n}$ relative to $B$ such that $\sum_n \leb \+ S_n < \infty$. At stage  $s$,  when  $x<s$ is least such that $A_s(x)\neq A_{s-1}(x) $,   list $V_{x,s}$ in~$\+ S$.      
This yields     a  Solovay test relative $B$ by the hypothesis that the approximation of~$A$ obeys~$\cost$.

Since $B$ is low for ML-randomness, the set $Y$ is ML-random relative to $B$. Hence $Y$ passes the Solovay test $\+ S$.  So choose $s_0$ such that $  Y \not \in V$ for any
$V$ listed in~$S$ after stage $s_0$. Given an input $x\ge s_0$, using~$Y$ as an oracle
  compute $t >x$ such that $[Y\uhr t] \sub V_{x,t}$.
We claim that  $A(x) = A_t(x)$. Otherwise  $A_{s}(x) \neq A_{s-1}(x)$ for some $s > t$, which
 would cause   $V_{x,s}$ (or some superset $V_{y,s}$, $y< x$) to be listed in~$G$, contrary to $Y\in
V_{x,t}$.

  \subsubsection{Proof of   (c) $\RA$ (d)} Note that each superlow set is $\omega$-c.a. So it suffices to show that no  $\omega$-c.a.\ set, and no superhigh set, is weakly Demuth random. For $\omega$-c.a.\ sets this is immediate from the definitions. For superhigh sets, this is a  result of \Kuc\ and Nies \cite[Cor.\ 3.6]{Kucera.Nies:11}.


\subsubsection{Proof of   (d) $\RA$ (a)}  For convenience we restate the   implication to be shown:

\n  {\it Let $\+ C$ be the class of superlow, or of superhigh sets.
 Suppose $A$ and $B$ are $K$-trivial c.e.\ sets.  If $A \leT Y \oplus B$ for each $Y \in \+ C \cap \MLR$, then  $A \le_{SJR} B$.} 
 
 We remark that for this implication, the hypothesis suffices   that $A$ is c.e.\ and superlow, and $B$ is Demuth traceable as discussed below; in particular, this holds if $B$ is   c.e.\ and superlow.

 As mentioned, the proof of this implication  combines methods from  Nies~\cite{Nies:11} and Bienvenu et al.~\cite{Bienvenu.Downey.ea:14}.  Below we will review and use  some    technical   notions   from these articles.
  First we consider   \cite[Def.\ 1.7]{Bienvenu.Downey.ea:14}. Given an oracle $B$, a $\DemBLR\la B \ra$ test generalises  a Demuth test $\seq {G_m}\sN m$  in that $G_m  \sub \cantor $ is a  $\SI 1 (B)$ class with $\leb G_m \le \tp{-m}$, and there is a  function $f$ taking $m$ to an index for $G_m$ such that $f$ has a $B$-computable approximation $g$, with  $g(m,t)$  having a    number of changes that is computably bounded in $m$. (Here we will only need the case that $f$ is $\omega$-c.a., in which case   the only difference to usual Demuth tests  is that the versions of the components are uniformly $\SI 1 (B)$,   rather than   $\SI 1$.)  

The authors in \cite{Bienvenu.Downey.ea:14} defined  that an oracle $B$ is low for $\DemBLR$ tests if every $\DemBLR\la B\ra$ test can be covered by a Demuth test, in the sense that passing  the Demuth test implies passing the $\DemBLR\la B\ra$ test.  In \cite[Thm. 1.8]{Bienvenu.Downey.ea:14} they characterized such oracles via a tracing condition called Demuth traceability.  For c.e.\ sets,  this  condition is equivalent  to being jump traceable, or again superlow~\cite[Prop.\ 4.3]{Bienvenu.Downey.ea:14}. 

 The  
  following lemma    holds for   any oracle $B$. We will prove it shortly.

  \begin{lemma} \label{lem:hDemtest}   For a  given order function~$h$   and a superlow c.e.\ set $A$, one  can build a $\DemBLR \la B \ra $ test  $(\mathcal H_m)\sN m$ such that, if  $A\leT Y\oplus B$ for some ML-random set~$Y$ passing this test, then  the function $J^A$ has a $B$-c.e.\ trace with  bound~$h$. \end{lemma}
 Nies \cite{Nies:11} defined  a class $\CCC \sub 2^\NN $ to be  \emph{Demuth test--compatible}  if 
  each  Demuth test is passed by a   member of $\CCC$. 
  Using   some methods from \cite{Greenberg.Hirschfeldt.ea:12}, he  showed in  \cite[Section 4]{Nies:11} that the  superlow ML-random sets, as well as the    superhigh ML-random sets,  are Demuth test--compatible.

 If a class $\+ C$ is Demuth test--compatible and $B$ is Demuth traceable, then each $\DemBLR\la B\ra$-test is passed by a set in $\+ C$.  So the lemma    suffices to establish the implication (d) $\RA$ (a) in question.

To verify   Lemma~\ref{lem:hDemtest},  let  $\Phi$ be a   Turing functional such that  $\Phi(0^e1 Y \oplus B)= \Phi_e(Y\oplus B)$   for each~$e,Y$. We  reduce the lemma to     the following.

	\begin{claim} \label{claim:DemYSm} There is  a $\DemBLR \la B \ra $     test $(\mathcal S_m)\sN m$ such that, if  $A= \Phi^{Y\oplus B}$ for some $Y$ passing this test, then   $J^A$ has a $B$-c.e.\ trace with  bound~$h$.  \end{claim} 
	
This claim  suffices to obtain Lemma~\ref{lem:hDemtest}: let  $(\mathcal H_m)\sN m$  be the $\DemBLR\la B \ra$  test obtained as in \cite[Lemma 2.6]{Nies:11}  applied to the test $(\mathcal S_m)\sN m$. Thus, if a set~$Y$ passes   $(\mathcal H_m)\sN m$, then $0^e1Y$ passes   $(\mathcal S_m)\sN m$ for each~$e$. 
By hypothesis of Lemma~\ref{lem:hDemtest},  $A\leT Y\oplus B$ for some $Y$ passing $(\mathcal H_m)\sN m$, so we have $A
		 =\Phi_e^{Y\oplus B}$ for some $e$, and hence $A= \Phi({0^e1Y}\oplus B)$. Since $0^e1Y$ passes   $(\mathcal S_m)\sN m$, we can conclude from the claim that $J^A$ has a $B$-c.e.\ trace with  bound~$h$.

It remains to prove   Claim~\ref{claim:DemYSm}. Write $\mathcal U_e= \Opcl {W^B_e}$. For the duration of the proof of the claim,    a sequence $\seq {G_m}$ of open sets will  be called an  \emph{$(A,B)$-special test}  if there is a Turing functional $\Gamma$ such that $G_m$ is the final version of the sets $G_m[t]= \+ U_{\Gamma^A(m,t)}$ over stages $t$, and  there  is a computable function~$g$ such that the number of changes of $\Gamma^A(m,t)$ is  bounded   by $g(m)$.  A set $Y$ passes such a test in the usual sense of Demuth tests, namely, $Y \not \in G_m$ for almost all $m$.

  We first observe that since  $A$ is superlow and c.e., there is a $\DemBLR\la B \ra$ test $\seq{\+ S_m}$ that covers  $\seq {G_m}$ in the sense  that each $Y$ passing $\seq{\+ S_m}$ passes $\seq {G_m}$. To see this, let $\Theta$ be a Turing functional such that $\Theta^X(m,i)$ is the $i$-th value of $\Gamma^X(m,i)$, for each oracle $X$. Note that by \cite[Lemma 2.7]{Nies:11}   there is a computable enumeration $(A_s)\sN s$ of $A$ and a computable function $f$ such that for at most $f(m,i)$ times a computation   $\Theta^{A_s}(m,i)$ is destroyed.  At stage $t$,  let 
\bc $\mathcal S_m[t]= \mathcal U_{\Theta^{A_t}(m,i)}$ \ec
where $i$ is maximal such that the expression on the right is defined at stage $t$. Clearly the number of times a version $\mathcal S_m[t]$ changes is bounded by $\sum_{i=0}^{g(m)} f(m,i)$.  Thus, $(\mathcal S_m)\sN m  $ is  a Demuth test. If an oracle $Y$ is in  $  G_m$ then $Y$ is in the final version $\mathcal S_m[t]$, so the new test indeed covers $\seq {G_m}$.  


So to establish Claim~\ref{claim:DemYSm} it suffices to build an $(A,B)$-special test $\seq {G_m}$ in place of $\seq{\+ S_m}$. To do so, we mostly follow~\cite[proof of Thm 3.2]{Nies:11}. For $m \in \NN$ let \bc $I_m= \{x\colon \, 2^m \le h(x) < 2^{m+1}\}$. \ec
  At stage $t$,  let  $u$ be  the maximum use of the computations $J^A(x)$ for $x  \in   I_m$ that exist.  We enumerate into the current version $G_m[t]$ all oracles $Z$ such that $\Phi_t^{Z\oplus B} \succeq A\uhr u$,  as long as the measure stays below $\tp{-m}$.  Whenever a new computation $J^A(x)$ for $x  \in   I_m$   converges at a stage, we start a new version of $ G_m$. Clearly, there will be  at most  $\#  I_m$ many  versions. 

More formally, there is a Turing functional $\Gamma$ such that for each string $\aaa$ of length~$t$, 
$$  \mathcal U_{\Gamma^\aaa(m,t)} = \{Z \colon \,  \fa x \in   I_m \  [ J^\aaa_t(x) \DA \text{with use} \, u \RA \aaa \uhr u \preceq \Phi^{Z\oplus B}_t ]\}. $$
Let $  G_m[t] = 
  \mathcal U_{\Gamma^A(m,t)}^{(\le \tp{-m})}$. Here for a $\SI 1(B)$ set $\+ W$ and rational $\epsilon >0$, as in the Cut-off Lemma~\cite[Lemma 2.1]{Nies:11} by $\+ W^{(\le \epsilon)}$ one denotes the $\SI 1(B)$ set given by the  enumeration capped by measure $\epsilon$.  By the uniformity of the Cut-off Lemma, from $m,t$ with the help of oracle $A$ we can compute an index for this  effectively open class. Thus,  the versions  $  G_m[t] $ define  an $(A,B)$-special    test $(  G_m)\sN m$.  

The  $B$-c.e.\ trace  $(T_x)\sN x$ is defined as follows.
At stage $t$, for each string~$\aaa$  of length $t$ such that $y = J_t^\aaa(x)$ is defined and the  measure of the current approximation to the c.e.\ open set $\mathcal U_{\Gamma^\aaa(m,t)}$ exceeds $\tp{-m}$,   put $y$ into $T_x$. The idea is that,  if $y=J^A(x)$,  then this must happen for some $\aaa \prec A$, otherwise~$Y $ can be put into $ G_m$ because there is no cut-off. 

The verification is similar to~\cite[proof of Thm 3.2]{Nies:11} \emph{mutatis mutandis}. We omit the proofs of the two claims that follow, which are similar to the claims in the proof of the corresponding result   \cite[Thm 3.2]{Nies:11}.

\n {\bf Claim 1.} {\it $(T_x) \sN x$ is a $B$-c.e.\ trace such that for each $x$ we have $\# T_x \le  h(x)$.} 
%

%

 \vsps
 
\n {\bf Claim 2.} {\it For almost every $x$, if $y= J^A(x)$ is defined,  then $y \in T_x$.}

This completes the proof of Claim~\ref{claim:DemYSm} and hence Lemma~\ref{lem:hDemtest}, and hence of  the implication in question.  

%


%
%
 
 \subsubsection{Proof of   (a) $\RA$ (b)}  This implication works in the context of much weaker assumptions on $A$ and $B$. We state this   separately. 
\begin{proposition}\label{lem:SJT_JT_to_obey}
If $A \le_{SJR} B$ and $B$ is jump traceable, then for every benign cost function $c$, we have $A \models_B c$.
\end{proposition}
 
First we need another lemma. 
\begin{lemma}
Suppose $T$ is a finite tree, and $v_0, \dots, v_{n-1} \in T$ are pairwise distinct, such that each $v_i$ has at least 2 children in $T$.  Then $T$ has at least $n+1$ leaves.
\end{lemma}

We omit the proof, which is a simple induction on $n$.

\begin{proof}[Proof of Proposition~\ref{lem:SJT_JT_to_obey}] Fix $c$ a benign cost function and $f$ an $\omega$-c.a.\ function with $c(f(n)) < 2^{-n}$ for all $n$.  We also denote by $f$ a computable approximation to $f$, so that $\lim_s f(n, s) = f(n)$ for all $n$, and such that for each $n$, $|\{ f(n, s) : s \in \omega\}| \le g(n)$, where $g$ is some total computable function.  We may assume that $f(n,s)$ is non-decreasing in both $n$ and $s$.  Fix $h$ such that $B$ is $h$-JT.

\smallskip

First we employ standard tricks to assume we already have the traces for the partial functions we intend to build.  We define sets $I_n^e$ and $J_n^e$ for $e < n < \omega$:
\begin{itemize}
\item For each $e$, the sets $\{ I_n^e : e < n < \omega\}$ partition $\omega^{[2e]}$ such that
\[
|I_n^e| = g(n) + \sum_{i < n} h(\seq{2e, n+1,i});
\]
\item For each $e$, the sets $\{ J_n^e : e < n < \omega\}$ partition $\omega^{[2e+1]}$ such that
\[
|J_n^e| = 2^{\sum\{ h(\seq{2e+1, x, i}) \ : \ x \in I_n^e \ \& \ i < n\}}.
\]
\end{itemize}
We then define $k$ an order function such that for all $e < n$ and all $x \in I^e_n \cup J^e_n$, $k(x) \le n$.

We will build partial functions $\Phi^A$ and $\Psi^B$.  Fix an effective listing of all c.e.\ $h$-traces and all oracle c.e.\ $k$-traces.  For $e = \seq{e_0, e_1}$, we will construct $\Phi$ and $\Psi$ on $\omega^{[2e]} \cup \omega^{[2e+1]}$ under the assumption that the $e_0$th element of the first listing traces $\Psi^B$ and the $e_1$th element of the second listing traces $\Phi^A$ with oracle $B$.  For the remainder of the construction, fix $e$, and let $(V_y)_{y \in\omega}$ and $(U_x^-)_{x \in \omega}$ be these elements, respectively.  We will drop the $e$ superscripts on the $I$ and $J$.

\smallskip

Suppose first that $(U_x)_{x \in \omega}$ were an oracle-free c.e.\ $k$-trace of $\Phi^A$.  We will have a module for each $n > e$.  Our module for $n$ seeks to test $A$ at various lengths, in particular at length $f(n)$.  At stage $s = n$, or at stage $s > n$ with $f(n, s) \neq f(n, s-1)$, the module will test length $f(n, s)$.  It will also test various lengths as they are provided to it by the $n+1$ module.

For a length $\ell$, it will first test $A\uhr{\ell}$ in an element of $x \in I_n$ -- that is, we define $\Phi^\sigma(x) =\sigma$ for all $\sigma \in 2^\ell$, and we monitor the strings enumerated into $U_x$.  This will narrow down the possibilities for $A\uhr{\ell}$ to a set of at most $n$ strings.  The module will then test each of those strings in $J_n$ (we will say more about how testing is done in $J_n$ in a moment).  If more than one of those strings were to pass this second test, we would {\em promote} $\ell$, i.e.\ tell the $n-1$ module that it is responsible for testing $A\uhr{\ell}$.  We will arrange that the $n$ module promotes at most $n-1$ lengths, so that the $n-1$ module will need to handle at most $g(n-1) + n-1$ lengths.

For $n = e+1$, the $n$ module will still declare lengths to be promoted, even though there is no $n-1$ module to promote them to.  This will not affect the running of the $n$ module, and so it will continue on as if there were an $n-1$ module.

Of course, $(U_x)_{x \in \omega}$ only traces $\Phi^A$ with oracle $B$, so we will need to rely on $(V_y)_{y \in \omega}$ to approximate this.  When we decide to test a length $\ell$ at $x \in I_n$, we will simultaneously define $\Psi^Y(\seq{2e+1,x, i})$ for $i < n$ and all oracles $Y$ to be the $i$th element enumerated into $U_x^Y$, if such an element exists.  Then
\[
A\uhr{\ell} \in U_x \subseteq 2^\ell \cap \bigcup_{i < n} V_{\seq{2e+1,x,i}},
\]
which has size at most $\sum\{ h(\seq{2e+1, x, i}) \ : \ i < n\}$.  We will test each of these strings on $J_n$.  So we will test at most $\sum\{ h(\seq{2e+1, x, i}) \ : \ x \in I_n^e \ \& \ i < n\}$ strings on $J_n$, which the reader will note is the exponent of the size of $J_n$.

The oracle $B$, using $(U_x^B)_{x \in \omega}$, will have an opinion as to which lengths should be promoted.  Again, we will arrange that $B$ sees at most $n-1$ lengths which are to be promoted by the $n$ module.  We define $\Psi^Y(\seq{2e, n, i})$ for $i < n-1$ and all oracles $Y$ to be the $i$th length which $Y$ believes the $n$ module should promote.  So the lengths $B$ believes should be promoted by the $n$ module will be elements of $\bigcup_{i < n-1} V_{\seq{2e, n, i}}$, which has size at most $\sum_{i < n-1} h(\seq{2e, n,i})$.  Thus the $n-1$ module will need to handle at most $g(n-1) + \sum_{i < n-1} h(\seq{2e, n,i})$, which the reader will note is the size of $I_{n-1}$.  So each $I_n$ is large enough to test every length the $n$ module must consider.

We must explain what it means to test a string on $J_n$.  We identify $J_n$ with $2^{\sum\{ h(\seq{2e+1, x, i}) \ : \ x \in I_n^e \ \& \ i < n\}}$, which we think of as a hypercube of side length 2 and dimension $\sum\{ h(\seq{2e+1, x, i}) \ : \ x \in I_n^e \ \& \ i < n\}$.  When we seek to test a string $\sigma$ on $J_n$, we choose an axis $d$ of the hypercube, split the hypercube into two pieces orthogonal to this axis, and define $\Phi^\sigma(x) = \sigma$ for all $x$ belonging to one of these pieces.  To that end, for each $d < \sum\{ h(\seq{2e+1, x, i}) \ : \ x \in I_n^e \ \& \ i < n\}$, let $J_n(d) = \{ \tau \in J_n : \tau(d) = 0\}$.  For each string $\sigma$ that we seek to test on $J_n$, we will choose a unique $d$ and define $\Phi^\sigma(x) = \sigma$ for each $x \in J_n(d)$ with $\Phi^\sigma(x)$ not already defined.  Since we will seek to test at most $\sum\{ h(\seq{2e+1, x, i}) \ : \ x \in I_n^e \ \& \ i < n\}$ many strings on $J_n$, there are sufficient $d$ to give each $\sigma$ a unique $d$.

Next, we must explain what it means to promote a length.  Recall that this is defined for each oracle $Y$, but we only care about it for oracle $B$.  At a stage $s$, let $\ell$ be the longest length which we have already decided should be promoted by the $n$ module (or $\ell = 0$ if there is no such length).  A string $\sigma$ which has been tested on $J_n$ is {\em confirmed at $n$} (relative to $Y$) if $\sigma \in U_{x,s}^Y$ for each $x$ with $\Phi^\sigma(x) = \sigma$.  If there are distinct $\sigma_0, \sigma_1$ of the same length which are both confirmed at $n$ at stage $s$, and such that $\sigma_0\uhr{\ell} = \sigma_1\uhr{\ell}$, then $|\sigma_0|$ is to be promoted by the $n$ module.  In particular, it must be that $|\sigma_0| > \ell$.

This completes the description of the original construction.

\smallskip

We must argue that for each $n > e$, $B$ believes at most $n-1$ lengths should be promoted by the $n$ module.  Suppose not, and fix lengths $0 = \ell_0 < \ell_1 < \ell_1 \dots < \ell_n$ such that for $i > 0$, $B$ believes $\ell_i$ is to be promoted by the $n$ module.  For each $i > 0$, fix strings $\sigma_0^i$ and $\sigma_1^i$ on the basis of which $B$ decided to promote $\ell_i$.  

By construction, $B$ decides $\ell_i$ should be promoted before it decides $\ell_{i+1}$ should be, and thus $\sigma_0^{i+1}\uhr{\ell_i} = \sigma_1^{i+1}\uhr{\ell_i}$ for $i > 0$.  Clearly this also holds for $i = 0$.  Now define the following sequence of sets:
\begin{itemize}
\item $Z_n = \{ \sigma_0^{n}, \sigma_1^{n}\}$;
\item For $0 < i < n$, 
\[
Z_i = Z_{i+1} \cup ( \{\sigma_0^i, \sigma_1^i\} \setminus \{ \sigma\uhr{\ell_i} : \sigma \in Z_{i+1}\}).
\]
\end{itemize}
Note that each $Z_i$ is an antichain, and $\{\sigma_0^i, \sigma_1^i\} \subseteq \{ \sigma \uhr{\ell_i} : \sigma \in Z_i\}$ by construction.  Let $T = \{ \sigma\uhr{\ell_i} : \sigma \in Z_1 \ \& \ i \le n\}$, which we think of as a tree.  Note that the leaves of $T$ are precisely $Z_1$.

For $i < n$, let $v_i = \sigma_0^{i+1}\uhr{\ell_i} = \sigma_1^{i+1}\uhr{\ell_i}$.  Then the $v_i$ are pairwise distinct and each has at least 2 children in $T$ (namely, $\sigma_0^{i+1}$ and $\sigma_1^{i+1}\uhr{\ell_i}$).   Thus $|Z_1| \ge n+1$.

Let $D$ be the set of axes chosen for various $\sigma \in Z_1$, and define $\tau \in J_n$ by
\[
\tau(d) = \left\{\begin{array}{cl}
0 & d \in D,\\
1 & d \not \in D.
\end{array}\right.
\]
Observe that $\tau \in J_n(d) \iff d \in D$.

\begin{claim}
For each $\sigma \in Z_1$, we make the definition $\Phi^\sigma(\tau) = \sigma$.
\end{claim}

\begin{proof}
By construction, we will make this definition so long as we have not already defined $\Phi^\sigma(\tau)$ to be something else.  But our actions for any string $\sigma' \not \in Z_1$ will never do this, as such $\sigma'$ will have an axis $d \not \in D$, and so will not seek to make a definition at $\tau$.  And $\sigma' \in Z_1 \setminus \{\sigma\}$ will not do this, as they will only seek to make a definition for $\Phi^{\sigma'}(\tau)$, and $\sigma'$ and $\sigma$ are incomparable as $Z_1$ is an antichain.
\end{proof}

As each of the $\ell_i$ is promoted, we have $Z_1 \subseteq U_\tau^B$, contradicting $|U_\tau^B| \le h(\tau) = n$.

\smallskip

Thus the construction can proceed.  The remainder of the argument is relative to $B$.

Fix $\hat{\ell}$ the longest length which $B$ believes the $e+1$ module should promote.  Nonuniformly fix $A\uhr{\hat{\ell}}$.  Let $L(n,s)$ be the set of lengths being tested by the $n$ module at stage $s$.  At a stage $s$, define a partial sequence $\sigma_n^s$ for $e \le n \le s$ recursively:
\begin{itemize}
\item $\sigma_e^s = A\uhr{\hat{\ell}}$;
\item Given $\sigma_n^s$, define $\sigma_{n+1}^s$ to be a string $\tau$ extending $\sigma_n^s$ with $|\tau| = \max L(n+1,s)$, and such that for each $\ell \in L(n+1, s)$, $\tau\uhr{\ell}$ is confirmed at $n+1$ by stage $s$, if such a string $\tau$ exists.
\end{itemize}

\begin{claim}
There is at most one possible choice for $\sigma_n^s$.
\end{claim}

\begin{proof}
For $n = e$, this is immediate.

For $n > e$, suppose there were two distinct strings $\tau_0$ and $\tau_1$ which are appropriate to pick for $\sigma_n^s$.  Fix $\ell \in L(n, s)$ least with $\tau_0\uhr{\ell} \neq \tau_1\uhr{\ell}$.  Then $\tau_0\uhr{\ell}, \tau_1\uhr{\ell}$ witness the promotion of $\ell$ at stage $s$, and $\ell > |\sigma_{n-1}^s|$, as $\tau_0$ and $\tau_1$ both extend $\sigma_{n-1}^s$.  This contradicts $|\sigma_{n-1}^s| = \max L(n-1, s)$ (or contradicts the definition of $\hat{\ell}$ if $n = e+1$).
\end{proof}

\begin{claim}
Let $\ell_n = \max \bigcup_s L(n,s)$ for $n > e$, and $\ell_e = \hat{\ell}$.  Then $A\uhr{\ell_n} = \lim_s \sigma_n^s$ for $n \ge e$.
\end{claim}

\begin{proof}
Induction on $n$.  The case $n = e$ is immediate.

For $n > e$, first observe that $\ell_{n-1}$ is either a length promoted by the $n$ module (and so eventually an element of $L(n,s)$) or is $f(n-1, s) \le f(n,s)$ for some $s$, and so is bounded by an element of $L(n,s)$.  Thus $\ell_{n-1} \le \ell_n$.

Now fix $s_0$ sufficiently large such that $\sigma_m^s = A\uhr{\ell_m}$ for all $m < n$ and $s \ge s_0$, and such that $L(n,s_0) = \bigcup_s L(n,s)$.  As $\Phi^A$ is traced by $(U_x^B)_{x \in \omega}$, there is a stage $s_1 \ge s_0$ such that each $A\uhr{\ell}$ for $\ell \in L(n, s_0)$ is confirmed at $n$.  Then $A\uhr{\ell_n}$ is a possible choice for $\sigma_n^s$ for every $s \ge s_1$, and thus is $\sigma_n^s$.
\end{proof}

Define a sequence of stages $(s_t)_{t \in \omega}$ as follows:
\begin{itemize}
\item $t_0 = e$.
\item Given $s_t$, $s_{t+1}$ is the least $s>s_t$ such that for every $n$ with $e < n \le t$, $\sigma_n^s$ exists. 
\end{itemize}
Define $A_t = \sigma_t^{s_t}$.

\begin{claim}
$(A_t)_{t \in \omega} \models c$
\end{claim}

\begin{proof}
Suppose $e < n \le t$ and $A_t(z) \neq A_t(z+1)$ for some $z$ with $c(z, t) \ge 2^{-n}$.  
As $c(z, s_t) \ge c(z, t)$, $z < f(n, s_t) \in L(n, s_t)$.  Thus $\sigma_n^{s_t} \neq \sigma_n^{s_{t+1}}$.  Fix $m$ least with $\sigma_m^{s_t} \neq \sigma_m^{s_{t+1}}$.  Fix $\ell \in L(m, s_t)$ least with $\sigma_m^{s_t}\uhr{\ell} \neq \sigma_m^{s_{t+1}}\uhr{\ell}$.   If no length less than $\ell$ and greater than $\max L(m-1, s_t)$ is promoted by the $m$ module at a stage $s \in (s_t, s_{t+1}]$, then these witness the promotion of $\ell$ at stage $s_{t+1}$, and $\ell > \max L(m-1, s_t)$, as both $\sigma_m^{s_t}$ and $\sigma_m^{s_{t+1}}$ extend $\sigma_{m-1}^{s_t} = \sigma_{m-1}^{s_{t+1}}$.  So whenever there is such a $z, n$ and $t$, there is a promotion by an $m$ module for $m \le n$ at a stage after $s_t$.

There can be at most $\sum_{e < m \le n} m-1 < n^2$ such promotions over the entire construction.  Thus we can bound
\[
c( (A_t)_{t \in \omega}) < \sum_n n^2 \cdot 2^{-n} < \infty.\qedhere
\]
\end{proof}
This ends the proof of Proposition~\ref{lem:SJT_JT_to_obey}.
\end{proof}

   \section{Nies and Stephan: Update on the SMB theorem for measures}

  The   purpose of this post is to provide   an example showing that the boundedness hypothesis in  \cite[Prop.\ 24]{Nies.Stephan:19}  is necessary. 
  
   We briefly review some background and notation.
Let     $\mathbb A^\infty$
denote the topological space of one-sided infinite sequences of
symbols in an alphabet $\mathbb A$.  Randomness notions etc.\ carry
over from the case that  $\mathbb A = \{0, 1\}$.
A measure $\rho$ on $\mathbb A^\infty$ is called  \emph{shift 
invariant} if $\rho ( G) = \rho (T^{-1}(G))$ for each open   (and
hence each measurable) set $G$. 
The \emph{empirical entropy}  of a measure  $\rho$ along $Z\in \mathbb
A^\infty$ is given by the sequence of   random variables  \[
h^\rho_n(Z ) = -\frac 1 n \log_{|\mathbb A|} \rho [Z\uhr n].\]   
A shift invariant measure $\rho$ on $\mathbb A^\infty$  is called 
\emph{ergodic} if   every  $\rho$ integrable function $f $ with $f
\circ T = f$  is constant $\rho$-almost surely.
The following  equivalent  condition can be easier to check: for any strings
 $u,v \in \mathbb A^*$,  
\[ \lim_N \frac 1 N \sum_{k=0}^{n-1} \rho ([u] \cap T^{-k}[v]) = \rho
[u] \rho [v]. \]
For ergodic $\rho$, the entropy $H(\rho)$ is defined as $\lim_n
H_n(\rho)$, where \[H_n(\rho) =  -\frac 1 n \sum_{|w| = n}  \rho [w]
\log \rho [w].\]
Thus, $H_n(\rho) = \mathbb E_\rho  h^\rho_n$ is the expected value
with respect to $\rho$.  
Recall  that by concavity of the logarithm function  and subadditivity of the entropy $H(X,Y) \le H(X) + H(Y)$,   the limit  exists and equals the infimum of the sequence. This limit is denoted $H(\rho)$, the entropy of $\rho$.

\begin{theorem}[SMB   theorem, e.g.~\cite{Shields:96}] Let  $\rho $ be an  ergodic   measure 
on the space $\mathbb A^\infty$. 
For $\rho$-almost every~$Z$   we have  $\lim_n h^\rho_n(Z) = H(\rho)$.
\end{theorem}

\begin{proposition}[\cite{Nies.Stephan:19}, Prop.\ 7.2] \label{prop: bound} Let  $\rho $ be a computable   ergodic        measure
on the space $\mathbb  A^\infty$ such that for some constant $D$,  
each $h_n^\rho$ is   bounded above by $D$. Suppose the measure $\mu$
is \ML\ a.c.\  with respect to $\rho$. 
Then  $\lim_n E_\mu |h^\rho_n - H(\rho)| =0$. 
\end{proposition} 

We now given an example showing that the boundedness hypothesis on the $h_n^\rho$ is necessary. In fact we  provide a computable  ergodic measure $\rho$ such that some finite measure $\mu \ll \rho$ makes the sequence $E_\mu h_n^\rho$ converge to $\infty$. This condition  $\mu \ll \rho$ (every $\rho$ null set is a $\mu$ null set) is stronger than  requiring that $\mu$
is \ML\ a.c.\  with respect to $\rho$.
\begin{example} \label{ex:unbounded}  There is an ergodic computable measure $\rho$ (associated to a binary renewal process) and a computable measure $\mu \ll  \rho$  such that 
$\lim_n \mathbb E_\mu h^\rho_n = \infty$. (We can then normalise $\mu $ to become a probability measure, and still have the same conclusion.)
\end{example}
%
\begin{proof} Let $k$ range over positive natural numbers.
The real  $c = \sum_k \tp{-k^4}$ is computable. Let $p_k =  \tp{-k^4}/c$ so that $\sum p_k=1$.   Let $b= \sum_k k \cdot p_k$ which is also computable.

Let $\rho $ be the measure associated with the corresponding binary renewal process, which is  given by the conditions \begin{center} $\rho[Z_0 =1]= 1/b$ and $\rho(10^k 1 \prec Z  \mid Z_0 =1) = p_k$.  \end{center}
Informally, the process has initial value  $1$ with probability $1/b$, and after each $1$ with probability $p_k$ it takes $k$ many 0s until it reaches the next 1. See again e.g.~\cite[Ch.\ 1]{Shields:96} where it is shown that $\rho$ is ergodic.
Write $v_k = 1 0^k 1$. Note that $\rho[v_k] = p_k/b$. 

Define a  function $f$ in $L_1(\rho)$ by $f(v_k \ape Z ) =   k^{-2}/p_k$ and $f(X)=0$ for any $X$ not extending any $v_k$.  
  It is clear that $f$ is $L_1(\rho)$-computable, in the usual sense that there is an effective sequence of basic functions $\seq{f_n}$ converging effectively to $f$: let $f_n(X)= f(X)$ in case $v_k \prec X$, $k \le n$, and $f_n(X)= 0 $ otherwise. Define   the    measure $\mu$   by  $d\mu = f d\rho$, i.e. $\mu(A) = \int_A f d\rho$.  Thus $\mu[v_k]= k^{-2}/b$. Since $\rho$ is computable and $f$ is $L_1(\rho)$-computable, $\mu$ is computable. Also note that $\mu(\cantor)= \int f d\rho$ is finite.

For any $n>2$, letting   $k= n-2$, we have
 \[E_\mu h_n^\rho \ge -\frac 1 n \mu[v_k] \log \rho[v_k]= - \frac 1 {n  k^{2} b}   (k^4-bc) \ge \frac {k^2} {nb} -O(1).\]
\end{proof}

\section{Nies and Tomamichel: the measure associated with an infinite sequence of qubits}

For background and notation see the 2017 Logic Blog entry~\cite[Section~6]{LogicBlog:17}.  

Recall that mathematically, a qubit is a  unit vector in the Hilbert space~$\mathbb C^2$. We  give  a brief summary  on   ``infinite  sequences" of qubits. One considers the $C^*$   algebra $M_\infty =  \lim_n M_{2^n}(\mathbb C)$, an  approximately finite  (AF) $C^*$ algebra. ``Quantum Cantor space" consists of the state set $\+ S(M_\infty)$, which is a convex, compact, connected set with a shift operator, deleting the first  qubit. 

Given a finite  sequence of qubits, ``deleting"  a particular one generally  results in a statistical superposition of the remaining ones. This is is why $\+ S(M_\infty)$ consists of coherent sequences  of density matrices $\mu = \seq {\mu_n}\sN n$ where $\mu_n$ is  in $M_{2^n}(\mathbb C)$ (density matrices formalise such superpositions), rather than just of sequences of unit vectors in $(\mathbb C^2)^{\otimes n}$.  To be coherent means that $T(\mu_{n+1})= \mu_n$ where $T$ is the partial trace operation deleting the last qubit. For more background on this, as well as  an algorithmic  notion of randomness   for such states,  see Nies and Scholz \cite{Nies.Scholz:18}.  

We defined in~\cite[Section~6]{LogicBlog:17} what it means for a state $\mu $ on $M_\infty$ to be qML-random with respect to a computable shift invariant state $\rho$. 

For each density matrix $D\in M_{2^n}$ its diagonal $\ol D$ is also a density matrix. This is because the operator $A \mapsto \sum_{x \in \{0,1\}^n} P_x D P_x$ is completely positive and trace preserving (here as usual  $P_x = | x \ra \la x |$ is the projection onto the subspace spanned by the basis vector given by $x$).   Clearly   $\ol \mu_n = T(\ol {\mu_{n+1}})$ for each $n$ because taking the partial trace means adding corresponding items on the two quadratic $\tp n \times \tp n$ components along the diagonal. So $\ol \mu$ is a diagonal state, and hence corresponds to a measure on Cantor space. 
Clearly if $\mu$ is computable then so is $\ol \mu$.  Shift invariance is also preserved by this operation. 

Ergodicity of $\ol \mu$ can be used as a test of the more complicated ergodicity for $\mu$.
\begin{fact} If $\ol \mu$ is  ergodic then so is $\mu$. \end{fact}
\begin{proof} Suppose  $\mu = \aaa \eta + \beta \nu$ is a nontrivial convex decomposition of $\mu$  into shift-invariant states. Then $\ol \mu = \aaa \ol \eta + \beta \ol \nu$ is a nontrivial convex decomposition  of $\ol \mu$  into shift-invariant states as well. \end{proof}

\begin{fact} Let $\rho$ be a  computable shift-invariant measure. If  a state $\mu$ is qML-random wrt $\rho$ then so is $\ol \mu$. \end{fact}
\begin{proof} Note that for each classical $\SI 1$ set $G$ we have $\mu(G) = \ol \mu (G)$, where on the left hand side $G$ is interpreted as an ascending sequence $\seq {p_m}$ of clopen projections $p_m \in M_m$, and then   $\mu(G) = \lim_m \mu(p_m)$. But $\mu(p_m ) = \mathrm{Tr} (\mu\uhr m p_m)= \mathrm{Tr} (\ol \mu\uhr m p_m)$ because $p_m$ is diagonal. \end{proof}

We obtain a partial quantum version of Prop.\ \ref{prop: bound}. This answers  one special  case of Conjecture 6.3 in \cite{LogicBlog:17} (unfortunately the roles of $\mu$ and $\rho$ are   exchanged there).  The boundedness hypothesis turned out to be  necessary by Example~\ref{ex:unbounded}, but was not present in the statement of the conjecture back then.
\begin{proposition}  \label{prop: bound Q} Let  $\rho $ be a computable   ergodic        measure
on the space $\mathbb  A^\infty$ such that for some constant $D$,  
each $h_n^\rho$ is   bounded above by $D$. Write $s= H(\rho)$. 
Suppose the state   $\mu$
is qML random  with respect to $\rho$. 
Then  $\lim_n \mathrm{Tr} (\mu \uhr n  |h^\rho_n - s I_{\tp n}|) =0$. 
\end{proposition} 
Here the function $h^\rho_n$ is viewed as defined on strings of length $n$, and in the expression above we identify it with the corresponding diagonal matrix in $M_n$.
\begin{proof}
Let $\ol \mu$ be the classical state (measure) such that $\ol \mu \uhr n$ is the diagonal of  $\mu\uhr n$, as above. By the fact above we have $\ol \mu \ll_{ML} \rho$ i.e., the non $\rho$-MLR bit sequences form a $\ol \mu $ null set.  Now  \bc 
$ \mathrm{Tr} (\mu \uhr n  |h^\rho_n - s I_{\tp n}|) =  \mathrm{Tr} (\ol \mu \uhr n  |h^\rho_n - s I_{\tp n}|) = E_\mu |h^\rho_n - s| $.  \ec
It now suffices to apply  Prop.\ \ref{prop: bound}.
\end{proof} 

If $\rho$ is i.i.d.\ then the boundedness condition  on the $h_n^\rho$  holds. This yields a new proof of \cite[Thm. 6.4]{LogicBlog:17} (first turning  the ergodic state $\rho$ into a classical state by applying a fixed unitary ``qubit-wise", as before).


 \part{Set theory}
 \section{Yu: perfect subsets of uncountable sets of reals}
 We make some     remarks on a recent result: 
\begin{theorem}[Hamel, Horowitz, Shelah \cite{Hamel.etal:19}]\label{theorem: HHS19}

\

\n
Assume $ZF+DC$. If  every uncountable   Turing invariant set of reals  has a perfect subset, then so has  every uncountable set of reals.
\end{theorem}

We obtained an improvement of the Theorem which was added  in Section~III of the most recent version of~\cite{Hamel.etal:19}.

\begin{theorem}[Yu \cite{Hamel.etal:19}]\label{theorem: Yu20}
Assume $ZF+AC_{\omega}$. For any analytic countable equivalence relation $E$,  if every uncountable   $E$-invariant set of reals has a perfect subset, then so has every uncountable set of reals.
\end{theorem}

{\bf Remark 1}: Actually $AC_{\omega}$ can be removed from Theorems \ref{theorem: HHS19}  and  \ref{theorem: Yu20}.  In the recursion theoretic proof of Theorem \ref{theorem: HHS19}, the first use of $AC_{\omega}$ is to prove that $[Q]_T\cap A$ is uncountable. But this is clearly unnecessary, since otherwise $Q\subseteq [[Q_T]\cap A]_T$ would be countable but without appealing $AC_{\omega}$ due to the uniformity.  

The second use of $AC_{\omega}$ is to prove that $Q_{e,i}\cap A$ is uncountable for some $e,i$. But if $Q_{e,i}\cap A$ is countable for all $e,i$, then the computation is uniform since  $Q_{e,i}\cap A= Q_{e,i}\cap P$ is a countable closed set.  $AC_{\omega}$ can be removed from Theorem \ref{theorem: Yu20} for  similar reasons.

{\bf Remark 2}: Ironically we need $AC_{\omega}$ to prove the conclusion   for every countable Borel equivalence relation since the Borelness implying $\mathbf{\Sigma^1_1}$-ness  requires $AC_{\omega}$.
But for most natural Borel countable equivalence relations, it seems $AC_{\omega}$ is unnecessary.

   
%

%
%
\def\cprime{$'$} \def\cprime{$'$}

%
%

\end{document}